\documentclass{amsart}
\usepackage{pstricks}
\usepackage{url}
\usepackage{lscape}
\usepackage{mathtools}
\usepackage[colorlinks=true,linkcolor=red,citecolor=blue]{hyperref}
\usepackage{color}
\usepackage[color,matrix,arrow, curve]{xy}
\usepackage{amsmath, amsthm, amsfonts, amssymb}
\usepackage{mathrsfs}

\DeclareMathOperator{\ob}{Ob}
\DeclareMathOperator{\Hom}{Hom}
\DeclareMathOperator{\mor}{Mor}
\DeclareMathOperator{\Rmap}{\mathbb{R}Map}
\DeclareMathOperator{\map}{Map}
\DeclareMathOperator{\ho}{Ho}
\DeclareMathOperator{\End}{End}
\DeclareMathOperator{\hoEnd}{HoEnd}

\DeclareMathOperator{\hoHom}{HoHom}

\DeclareMathOperator{\REnd}{\mathbb{R}End}

\newcommand{\hocolimsub}[1]{\begin{array}[t]{cc} \textup{hocolim} \\
[-1.2mm] \scriptstyle{#1} \end{array}}

\DeclareMathOperator{\colim}{colim}
\renewcommand{\lim}{\mathrm{lim}}

\newcommand{\cob}{\mathsf{CoB}}
\newcommand{\corb}{\mathsf{CoRB}}
\newcommand{\pab}{\mathsf{PaB}}
\newcommand{\parb}{\mathsf{PaRB}}
\renewcommand{\bmod}{\overline{\mathcal{M}}}

\newcommand{\pb}{\mathsf{PB}}
\newcommand{\prb}{\mathsf{PRB}}
\newcommand{\rb}{\mathsf{RB}}
\newcommand{\br}{\mathsf{B}}
\newcommand{\ugt}{\widehat{\underline{\mathrm{GT}}}}
\newcommand{\gt}{\widehat{\mathrm{GT}}}

\newcommand{\Set}{\mathbf{Set}}

\renewcommand{\S}{\mathbf{S}}
\newcommand{\G}{\mathbf{G}}
\newcommand{\Op}{\mathbf{Op}}
\newcommand{\Grp}{\mathbf{Grp}}
\newcommand{\Pro}{\mathrm{Pro}}

\newcommand{\bM}{\mathbb{M}}

\newcounter{dummy} \numberwithin{dummy}{section}

\newtheorem{thm}[dummy]{Theorem}
\newtheorem{prop}[dummy]{Proposition}
\newtheorem{lemma}[dummy]{Lemma}
\newtheorem{cor}[dummy]{Corollary}
\newtheorem*{thm*}{Theorem}

\theoremstyle{definition}
\newtheorem{definition}[dummy]{Definition}
\newtheorem{cons}[dummy]{Construction}
\newtheorem{example}[dummy]{Example}
\newtheorem{remark}[dummy]{Remark}

\begin{document}
\title[Operads of genus zero curves and the GT group]{Operads of genus zero curves and the Grothendieck-Teichm\"{u}ller group}
\author{Pedro Boavida de Brito, Geoffroy Horel and Marcy Robertson}

\address{Dept. of Mathematics, Instituto Superior Tecnico\\ University of Lisbon \\ Lisboa, Portugal.}
\email{pedrobbrito@tecnico.ulisboa.pt} 
\address{LAGA, Institut Galil\'{e}e, Universit\'{e} Paris 13 \\ 99 avenue Jean-Baptiste Cl\'ement\\ 93430 Villetaneuse, France.}
\email{horel@math.univ-paris13.fr}
\address{School of Mathematics and Statistics \\ The University of Melbourne \\ Victoria 3010 \\ Australia}\email{marcy.robertson@unimelb.edu.au}

\begin{abstract}
We show that the group of homotopy automorphisms of the profinite completion of the genus zero surface operad is isomorphic to the (profinite) Grothendieck-Teichm\"{u}ller group. Using a result of Drummond-Cole, we deduce that the Grothendieck-Teichm\"{u}ller group acts nontrivially on $\bmod_{0,\bullet+1}$, the operad of \emph{stable} curves of genus zero. As a second application, we give an alternative proof that the framed little $2$-disks operad is formal. 

\end{abstract}

\maketitle

\section{Introduction} 

The moduli space of genus $g$ curves with $n$ marked points, $\mathcal{M}_{g,n}$, is defined over $\mathbb{Q}$ and as such its geometric fundamental group 
\[\pi_1^{geom}(\mathcal{M}_{g,n}):=\pi_1^{et}(\bar{\mathbb{Q}} \times_{\mathbb{Q}} \mathcal{M}_{g,n})\]
has an action of the absolute Galois group $\mathrm{Gal}(\bar{\mathbb{Q}}/\mathbb{Q})$. In \cite{groth}, Grothendieck proposed to study $\mathrm{Gal}(\bar{\mathbb{Q}}/\mathbb{Q})$ via its action on the geometric fundamental groups of \emph{all} the stacks $\mathcal{M}_{g,n}$ and the natural maps relating these for various $g$ and $n$. This collection he called the ``Teichm\"uller tower''.

A motivation for this idea was a theorem of Belyi's (see \cite{Belyi}) which implies that the action of $\mathrm{Gal}(\bar{\mathbb{Q}}/\mathbb{Q})$ on the geometric fundamental group $\pi_1^{geom}(\mathcal{M}_{0,4}) \cong
\widehat{F}_2$ is faithful. It follows that the absolute Galois group acts faithfully on the whole Teichm\"uller tower and it is an open question whether there are other automorphisms. An appealing aspect of this program is that the Teichm\"uller tower is a purely topological object since the geometric fundamental group of $\mathcal{M}_{g,n}$ is also the profinite completion of the mapping class group $\Gamma_{g,n}$ of a genus $g$ surface with $n$ marked points. In this way, Grothendieck's proposal creates a remarkable and unexpected bridge between number theory and low dimensional topology.

At the genus zero level, Ihara showed that the image of the action of $\mathrm{Gal}(\bar{\mathbb{Q}}/\mathbb{Q})$ on the geometric fundamental group of $\mathcal{M}_{0,4}$ lies in an explicitly defined profinite group $\gt$ constructed by Drinfel'd and called the Grothendieck-Teichm\"uller group. It is a longstanding problem to determine whether this injection $\mathrm{Gal}(\bar{\mathbb{Q}}/\mathbb{Q})\to\gt$ is an isomorphism. Given this, and granting the hypothetical relation between the absolute Galois group and the Teichm\"uller tower, one may wonder whether there is a relation between $\gt$ and this tower.

In this paper, we show that $\gt$ is the group of homotopy automorphisms of the genus zero Teichm\"uller tower. We do so for an \emph{operadic} definition of the \emph{genus zero Teichm\"uller tower} which encodes the natural relations between curves, as observed by Grothendieck. To define the genus zero Teichm\"uller tower, we replace \emph{marked points} by \emph{boundary components}. More precisely, we replace the group $\Gamma_{0,n}$ by $\Gamma_{0}^n$, the mapping class group of a genus zero Riemann surface with $n$ boundary components. This is not a drastic change since there is an isomorphism $\Gamma_0^n\cong \Gamma_{0,n}\times\mathbb{Z}^{n}$ obtained by collapsing boundary components to points. The advantage, however, is that the collection of spaces $\{B\Gamma_{0}^{n+1} \}_{n\geq 0}$ supports a much richer algebraic structure because of this small change. Indeed, two Riemann surfaces with boundary can be glued together along one of their boundary components. This gives us composition maps
\[B\Gamma_{0}^{n+1}\times B\Gamma_{0}^{m+1}\to B\Gamma_{0}^{n+m}.\]
If we replace the group $\Gamma_0^{n+1}$ by a certain homotopy equivalent groupoid $\mathcal{S}(n)$, we obtain the structure of an operad on the collection of spaces $\{B\mathcal{S}(n)\}$. This means that the composition maps satisfy associativity, $\Sigma$-equivariance and unital conditions. This object is denoted $\mathcal{M}$ and called the genus zero surface operad (we refer the reader to Definition~\ref{def:surface_operad} for a more precise definition). It is a suboperad of the operad constructed in \cite[Definition 2.3]{tillmann} and it is equivalent to the classical framed little $2$-disks operad (Proposition ~\ref{wahlprop}). Our first main theorem (Theorem \ref{mainthm}) can then be stated as follows:

\begin{thm*}
The group $\gt$ is isomorphic to the group of homotopy automorphisms of the profinite completion of the genus zero surface operad.
\end{thm*}

Apart from the gluing along boundary components, there are also natural maps $\Gamma_{0}^{n+1} \to \Gamma_0^{n}$
for $n \geq 0$, corresponding to filling in boundary components (i.e. extending diffeomorphisms to the missing disks by the identity). We prove in section \ref{sec:unital} a variant of the above theorem in which we incorporate these operations to the genus zero surface operad. This does not affect the conclusion and the group of automorphisms remains $\gt$.

The theorem implies that there is a faithful action of $\mathrm{Gal}(\bar{\mathbb{Q}}/\mathbb{Q})$ on the profinite completion of the genus zero surface operad. If we look at this result aritywise, this is not particularly surprising as the group $\widehat{\Gamma}_{0,n+1}\times\widehat{\mathbb{Z}}^{n+1}$ has an obvious action of $\gt$ that is faithful when $n$ is at least $3$. The difficulty in this theorem is to show that this action of $\gt$ is compatible with the operad structure and accounts for all the operad automorphisms. A corollary of this result is a new proof of the formality of the operad $\mathcal{M}$, or equivalently, the operad of framed little disks (see Section~\ref{section:formality}).

The compatibility of the $\gt$ action with the operad structure is somewhat related to a result of Hatcher, Lochak and Schneps \cite{hatchertower} that a certain subgroup of $\gt$ acts on the collection of the profinite completion of the pure mapping class groups $\Gamma_{g,n}^m$ of a genus $g$ Riemann surface with $n$ punctures and $m$ boundary components.  They show that this action is compatible with certain geometric operations relating these groups: it preserves conjugacy classes of Dehn twists along closed embedded curves and it is compatible with the operation of extending a diffeomorphism on a subsurface by the identity. 

When dealing with profinite completions, an important observation is that the profinite completion of an operad in spaces is no longer an operad, but rather an operad ``up to homotopy''. This technical detail requires the use of $\infty$-operads and is the topic of Sections \ref{sec:operads} and \ref{sec:pro-operads}.  Another key step in the proof is to replace the genus zero surface operad by a homotopy equivalent one: the parenthesized ribbon braid operad. This operad has a very combinatorial description that makes computing automorphisms more practical.

Returning to $\mathcal{M}_{0,n}$, the gluing of curves along marked points creates a nodal singularity and so the collection $\{\mathcal{M}_{0,n}\}$ does not form an operad. This can be fixed if we allow curves with singularities, i.e. if we replace the schemes $\mathcal{M}_{0,n}$ by their compactifications $\bmod_{0,n}$, the moduli spaces of \emph{stable} genus zero curves with $n$ marked points. Indeed, the collection of spaces $\{\bmod_{0,n+1}\}$ has the structure of an operad; the composition maps
\[\bmod_{0,n+1}\times\bmod_{0,m+1}\to\bmod_{0,n+m},\]
are obtained by gluing curves (possibly with nodal singularities) along marked points.

The genus zero surface operad $\mathcal{M}$ maps to the operad $\{\bmod_{0,\bullet+1}\}$. By a theorem of Drummond-Cole, the latter can be seen as an operadic quotient of the former by homotopically killing the circle in arity $1$. Our second main theorem (Theorem \ref{thm: action on bmod}) is the following:

\begin{thm*}
The action of $\gt$ on the profinite completion of the genus zero surface operad extends to an action of $\gt$ on the profinite completion of $\bmod_{0,\bullet +1}$. Moreover this action is non-trivial. 
\end{thm*}

We point out that the complex analytic spaces underlying the schemes $\bmod_{0,n}$ are simply connected, and so the geometric fundamental groups of these schemes are trivial. This deviates from most of the literature, where $\gt$ actions are usually constructed on schemes whose associated complex analytic spaces are $K(\pi,1)$'s. 

Finally, there is also a standard action of the absolute Galois group of $\mathbb{Q}$ on the profinite completion of the operad $\bmod_{0,\bullet +1}$, coming from the fact that this operad can be obtained as the geometric \'etale homotopy type of an operad in $\mathbb{Q}$-schemes. It seems plausible that the $\gt$ action that we construct coincides with the action of the absolute Galois group of $\mathbb{Q}$ restricted along the injection $\mathrm{Gal}(\bar{\mathbb{Q}}/\mathbb{Q})\to \gt$.

\subsection*{Acknowledgements} We are grateful to the Hausdorff Institute of Mathematics for the excellent working conditions during the Junior Trimester in Topology. P.B. was supported by FCT through grant SFRH/BPD/99841/2014. We would also like to thank Benjamin Collas, Philip Hackney and Craig Westerland for helpful mathematical discussions and Rosona Eldred and David Gepner for comments on earlier drafts of this paper.

\section{Homotopical recollections}

This background section serves as a brief overview of the homotopical constructions that we use in the 
paper. Throughout we use the language of Quillen model categories, and take \cite{hirschhorn} as our standard reference. We also make a mild but essential use of the vantage point of $\infty$-categories (via relative categories, recalled below). Throughout we use the term \emph{space} to mean simplicial set. 

\subsection{Relative categories and derived mapping spaces}

A \emph{relative category} is a pair $(\mathbf{C},\mathbf{W})$ where $\mathbf{C}$  is a category and $\mathbf{W}$ is a wide subcategory of $\mathbf{C}$ whose arrows we will call \emph{weak equivalences} in $\mathbf{C}$. A \emph{relative functor} $F:(\mathbf{C},\mathbf{W}) \to (\mathbf{D},\mathbf{W}^\prime)$ is a functor $F:\mathbf{C}\to \mathbf{D}$ such that $F(\mathbf{W}) \subset \mathbf{W}^\prime$. The \emph{homotopy category} of $\mathbf{C}$, denoted $\ho\mathbf{C}$, is the category obtained from $\mathbf{C}$ by formally inverting the maps in $\mathbf{W}$. 

The homotopy category of $\textbf{C}$ does not capture all of the higher order homotopical information contained in the relative category. As a homotopical enhancement for $\ho\mathbf{C}$, Dwyer-Kan \cite{DK2} constructed a simplicial category $L (\mathbf{C},\mathbf{W})$, together with a natural embedding $\mathbf{C} \to L(\mathbf{C},\mathbf{W})$, with the property that the category of components of $L (\mathbf{C},\mathbf{W})$ (i.e. the category obtained by applying $\pi_0$ to the morphism spaces of $L(\mathbf{C},\mathbf{W})$) agrees with $\ho \mathbf{C}$.

The simplicial category $L(\mathbf{C},\mathbf{W})$ has the same objects as $\mathbf{C}$ and for any two objects $X,Y$ in $\mathbf{C}$ a space of maps $\Rmap_\mathbf{C}(X,Y)$.  We will write $\Rmap(X,Y)$ if $\mathbf{C}$ is understood. One of the important features of $L(\mathbf{C},\mathbf{W})$ is that its morphism spaces are homotopically meaningful; that is, for any $Y$ in $\mathbf{C}$ and weak equivalence $X \to X^\prime$, the induced maps 
\[\Rmap(X^\prime, Y) \to \Rmap(X,Y)\]
and 
\[\Rmap(Y, X) \to \Rmap(Y,X^\prime)\]
are weak equivalences of spaces. Calculating $\Rmap$ for an arbitrary relative category is often not a feasible task. The situation simplifies if the relative category comes equipped with extra structure. If $\mathbf{C}$ has the extra structure of a simplicial model category, then we have the following.

\begin{thm}[{\cite[Corollary 4.7]{DK3}}]\label{rigidify}~
Let $\mathbf{C}$ be a simplicial model category, and let $X$ and $Y$ be objects in $\mathbf{C}$ such that $X$ is cofibrant and $Y$ is fibrant. Denote by $\map_{\mathbf{C}}(X,Y)$ the space of maps in $\mathbf{C}$ coming from the simplicial structure. Then $\Rmap_\mathbf{C}(X,Y)$ is related to $\map_{\mathbf{C}}(X,Y)$ by a natural zigzag of weak equivalences. 
\end{thm}  

\subsection{Adjunctions} A simplicial Quillen pair $(L,R)$ between simplicial model categories $\mathbf{C}$ and $\mathbf{D}$ gives rise to a homotopy adjunction 
\[
\mathbb{L}L : L(\mathbf{C},\mathbf{W}) \leftrightarrows L(\mathbf{D}, \mathbf{W}^\prime) : \mathbb{R}R
\]
of simplicially enriched categories in the sense that there is a weak equivalence of spaces
\[
\Rmap(\mathbb{L} L(X), Y) \simeq \Rmap(X, \mathbb{R} R(Y))
\]
natural in $X \in \mathbf{C}$ and $Y \in \mathbf{D}$, where $\mathbb{L}$ and $\mathbb{R}$ denote, respectively, the left and right derived functor constructions. That is to say, given a cofibrant replacement $X_c \xrightarrow{\sim} X$ of $X$ and a fibrant replacement $Y \xrightarrow{\sim} Y_f$, we obtain a weak equivalence
\[
\map_\mathbf{D}(L(X_c), Y_f) \simeq \map_{\mathbf{C}}(X_c, R(Y_f))
\]
of non-derived mapping spaces.

\subsection{Spaces and groupoids} 
We write $\S$ for the category of simplicial sets with its usual Kan-Quillen simplicial model structure. We write $\G$ for the category of groupoids with a model structure in which weak equivalences are equivalences of categories, cofibrations are morphisms that are injective on objects, and fibrations are isofibrations. Isofibrations are those functors which have the right-lifting property against the map $[0] \to E$ where $[0]$ denotes the trivial category with a single object $0$ and $E$ denotes the groupoid with two objects $0$ and $1$ and exactly two non-identity morphisms $0 \to 1$ and $1 \to 0$. In particular, every object in $\G$ is both fibrant and cofibrant. For details, see \cite[Section 5]{Anderson}.

The relationship between $\S$ and $\G$ is classical: the classifying space functor $$B : \G \to \S$$ has a left adjoint $\pi$ which assigns to a space $X$ its fundamental groupoid $\pi X$.  The model category structure on $\G$ is simplicial with mapping space given by $\map(C,D) := \map(BC,BD)$ for every pair of groupoids $C$ and $D$. The pair $(\pi,B)$ then forms a simplicial Quillen pair. The classifying space functor preserves and reflects all weak equivalences and fibrations and is homotopically fully faithful in the sense that the natural map
\[
\Rmap(C,D) \to \Rmap(BC, BD) 
\]
is a weak equivalence for every pair of groupoids $C$ and $D$.

\subsection{Operads}
 A \emph{symmetric sequence} in spaces is a sequence of spaces $\{ \mathsf{P}(n) \}_{n \geq 0}$, in which each space $\mathsf{P}(n)$ is equipped with an action of the symmetric group $\Sigma_n$. An \emph{operad} in spaces $\mathsf{P}$ is a symmetric sequence $\{ \mathsf{P}(n) \}_{n \geq 0}$ together with composition maps
\[
\circ_i : \mathsf{P}(n) \times \mathsf{P}(m) \to \mathsf{P}(n + m - 1)
\]
for $i \in \{1, \dots, m\}$ which are compatible with the symmetric group actions and subject to associativity and unit axioms. A map of operads is a map of symmetric sequences which preserves the operadic structure. Substitution of the word \emph{space} for the word \emph{groupoid} gives us the notion of an operad in groupoids. 

We denote by $\Op(\S)$ and $\Op(\G)$ the category of operads in $\S$ and $\G$, respectively. These categories are equipped with simplicial model category structures in which weak equivalences and fibrations are defined levelwise. More explicitly, a map of operads $f:\mathsf{P}\rightarrow\mathsf{Q}$ is a weak equivalence (respectively, fibration) if $f(n):\mathsf{P}(n)\rightarrow \mathsf{Q}(n)$ is a weak equivalence (respectively, fibration) for each non-negative integer $n$. For more details, consult \cite[Theorem 3.1]{bm03}.

Both the classifying space functor $B$ and the fundamental groupoid functor $\pi$ preserve products, and so induce an adjunction
\[
\pi : \Op(\S) \leftrightarrows \Op(\G) : B
\]
by levelwise application. It follows from \cite[Theorem 4.7]{bm07} that this is a simplicial Quillen adjunction. Moreover, $B$ is homotopically fully faithful.

The functor which to a groupoid $G$ associates its set of objects $\ob(G)$ is product preserving, and hence induces a functor $\ob$ from the category of operads in groupoids to the category of operads in sets.

An operad $\mathsf{P}$ is fibrant if each object $\mathsf{P}(n)$ is fibrant. In practice, it is more difficult to tell if an operad is cofibrant. However, for operads in groupoids we have the following useful criterion.  

\begin{prop}\cite[Proposition 6.8]{Horel1}\label{prop: free op gp}
The cofibrations in $\Op(\mathbf{G})$ are morphisms of operads in groupoids $f:\mathsf{P}\rightarrow\mathsf{Q}$ with the property that $\ob(f)$ has the left lifting property with respect to operad maps which are levelwise surjective. In particular, any operad $\mathsf{P}$ in groupoids with the property that $\ob(\mathsf{P})$ is free as an operad in $\Set$ is cofibrant in $\Op(\mathbf{G})$. 
\end{prop}

\section{Profinite Completion} 

Given any small category $\mathbf{C}$, the associated \emph{category of pro-objects} in $\mathbf{C}$ (alias pro-category of $\mathbf{C}$), $\Pro(\mathbf{C})$, is obtained by freely adding all cofiltered limits to $\mathbf{C}$. Formally, the opposite of $\Pro(\mathbf{C})$ is the full subcategory of the category of functors from $\mathbf{C}$ to $\Set$ spanned by those which are filtered colimits of representables. If $\mathbf{C}$ has finite limits, then the opposite of $\Pro(\mathbf{C})$ is also equivalent to the category of finite limit preserving functors from $\mathbf{C}$ to $\Set$. 

Alternatively, $\Pro(\mathbf{C})$ is the category whose objects are pairs $(I,X)$ where $I$ is a cofiltered category and $X=\{X_i\}_{i\in I}$ is a diagram $I\to \mathbf{C}$. Morphisms are defined as 
\[\Hom_{\Pro(\mathbf{C})}(\{X\}_{i}, \{Y\}_{j}) = \lim_{j\in J}\colim_{i\in I}\Hom_{\mathbf{C}}(X_i,Y_j) \; . \]
Clearly, $\mathbf{C}$ embeds fully faithfully in $\Pro(\mathbf{C})$. 

\begin{example}\label{expl:pro fin}
Let $\mathrm{Fin}$ be the category of finite sets. The category of profinite sets $\widehat{\Set}:=\Pro(\mathrm{Fin})$ is the associated pro-category. The category of profinite sets is equivalent to the category of compact, totally disconnected Hausdorff spaces and continuous maps. There is an adjunction
\[
\widehat{(-)}: \Set \leftrightarrows \widehat{\Set} : |-|
\]
where the right adjoint sends a diagram to its limit in $\Set$. The left adjoint sends a set $X$ to the diagram $R \mapsto X/R$ where $R$ runs over all equivalence relations on $X$ with finitely many equivalence classes. Another description of $\widehat{X}$ is as the finite limit preserving functor $\mathrm{Fin} \to \Set$ which sends a finite set $F$ to $\Hom(X,F)$.
\end{example}

\begin{example} 
The category of profinite groups $\widehat{\Grp}$ is the category of pro-objects in the category of finite groups. This category is equivalent to the category of group objects in $\widehat{\Set}$ (see, e.g., \cite[p. 237]{profinite_ref}). In other words, the category of profinite groups is the category of topological groups whose underlying topological space is a totally disconnected, compact Hausdorff space. There exists an adjunction 
\[
\widehat{(-)}:  \Grp \leftrightarrows \widehat{\Grp}:|-| \]
where the right adjoint sends a profinite group to the underlying discrete group. The left adjoint $\widehat{(-)}$ is called \emph{profinite completion}. It sends a group $G$ to the inverse limit of the diagram
\[N \to G/N\]
where $N$ runs over normal subgroups of $G$ with finite index and $G/N$ is given the discrete topology. 
\end{example} 

\begin{definition} 
We say that a groupoid $A$ is \emph{finite} if it has finitely many morphisms (and so also finitely many objects). The category of all finite groupoids will be denoted $f\G$. The associated pro-category is called the category of \emph{profinite groupoids}. It will be denoted by $\widehat{\G}:=\Pro(f\G)$.
\end{definition}

\begin{definition} Let $A$ be a profinite groupoid. Let $S$ be any finite set and $G$ be any finite group, then we define:
\begin{itemize} 
\item $H^0(A,S):= \Hom_{\widehat{\G}}(A,S)$. 
\item $Z^{1}(A,G):= \Hom_{\widehat{\G}}(A, *//G)$, where $*//G$ denotes the groupoid with a unique object whose group of automorphisms is the group $G$. 
\item $B^{1}(A,G):= \Hom_{\widehat{\G}}(A,G)$ where $G$ denotes the group $G$ seen as a discrete groupoid (i.e. with only identity morphisms).
\item $H^1(A,G):=Z^1(A,G)/B^1(A,G)$ where the quotient is taken with respect to a certain right action of the group $B^1(A,G)$ on $Z^1(A,G)$ (see \cite[Definition 4.1.]{Horel1} for more details).
\end{itemize} 
\end{definition}

The following two results are proved in \cite{Horel1}.

\begin{thm} The category $\widehat{\G}$ admits a left proper, cocombinatorial model structure in which a map $A\rightarrow B$ is a weak equivalence in $\widehat{\G}$ if 
\begin{enumerate}
\item for all finite sets $S$, $H^0(B,S)\to H^0(A,S)$ is an isomorphism and
\item for all finite groups $G$, $H^1(B,G)\to H^{1}(A,G)$ is an isomorphism. 
\end{enumerate}
The cofibrations are the maps which are monomorphisms on objects.
\end{thm} 

As with groups, there exists an adjunction 
\begin{equation}\label{adjunction_profinite_group}
\widehat{(-)}:  \G \leftrightarrows \widehat{\G}:|-| 
\end{equation}
in which the right adjoint sends a profinite groupoid, seen as diagram of groupoids, to its limit in $\G$. The left adjoint $\widehat{(-)}:  \G\to\widehat{\G}$ is called \emph{profinite completion}.

\begin{prop}\cite[Proposition 4.22]{Horel1} The profinite completion functor $\widehat{(-)}:  \G\to \widehat{\G}$ is a left Quillen functor. 
\end{prop}

\begin{remark}\label{rem:prod gpoids}
As a left adjoint, the profinite completion functor should not be expected to preserve limits. However, it does preserve certain products. More precisely, suppose $A$ and $B$ are two groupoids with finitely many objects. Then $\widehat{A \times B}$ is isomorphic to $\widehat{A} \times\widehat{B}$. This fact appears in \cite[Proposition 4.23.]{Horel1}.
\end{remark}

Recall that $\S$ denotes the category of simplicial sets. We always consider this category as a simplicial model category with the Kan-Quillen model structure. We denote the underlying relative category by $\mathcal{S}$. 

The definition of a profinite space and the homotopy theory of such is more involved than that of profinite groupoids. We begin by describing an $\infty$-categorical incarnation of the category of profinite spaces. A space is said to be\emph{ $\pi$-finite} if it has finitely many components and finitely many non-trivial homotopy groups, each of which is finite. The category $\mathcal{S}_{\pi-fin}$ of $\pi$-finite spaces forms a relative subcategory of $\mathcal{S}$. It has finite homotopy limits since a homotopy pullback of $\pi$-finite spaces is $\pi$-finite. We can thus form $\mathrm{Pro}(\mathcal{S}_{\pi-fin})$ which is an $\infty$-category with all limits. The pro-category of an $\infty$-category is defined similarly to the $1$-categorical case but replacing $\Set$ with $\mathcal{S}$ (see, e.g., around \cite[7.1.6.1]{LurieHT}). We view $\mathrm{Pro}(\mathcal{S}_{\pi-fin})$ as a homotopical enhancement of Example \ref{expl:pro fin} and call it the \emph{$\infty$-category of profinite spaces}. 

A presentation of this $\infty$-category as a model category is given by Quick in \cite{quick}.  Denote the category of simplicial objects in profinite sets $\mathrm{Fun}(\Delta^{op},\widehat{\Set})$ by $\widehat{\S}$. Quick equips the category $\widehat{\S}$ with a model structure in which the cofibrations are the monomorphisms and the weak equivalences are those maps that induce isomorphisms on $\pi_0$, $\pi_1$ and on twisted cohomology with finite coefficients. This model structure is cocombinatorial. It is proved in \cite[Corollary 7.4.8]{barneapro} that its underlying $\infty$-category models the $\infty$-category of profinite spaces.

The adjunction $\Set\leftrightarrows \widehat{\Set}$ between sets and profinite sets induces an adjunction:
\begin{equation}\label{adjunction profinite completion}
\widehat{(-)}:\S\leftrightarrows\widehat{\S}:|-|
\end{equation}
which Quick shows to be a Quillen adjunction in \cite[Proposition 2.28]{quick}.

The following definition is due to Serre.

\begin{definition}
A discrete group $G$ is said to be \emph{good} if for any finite abelian group $M$ equipped with a $G$-action, the map $G\to \widehat{G}$ induces an isomorphism
\[H^i(\widehat{G},M)\xrightarrow{\cong} H^i(G,M).\]
\end{definition}

\begin{prop}\label{prop:prodcompl}
Let $X$ and $Y$ be two connected spaces whose homotopy groups are good. Then the map $\widehat{X \times Y} \to \widehat{X} \times \widehat{Y}$ is a weak equivalence of profinite spaces.
\end{prop}
\begin{proof}
The profinite completion of a finite product of groups is the product of the profinite completions, and so the profinite completion of $\pi_i (X \times Y)$ is isomorphic to 
$$\widehat{\pi_i X} \times \widehat{\pi_i Y} \; .$$

For a connected space $Z$ whose homotopy groups are good, the profinite completion of $\pi_i Z$ is isomorphic to $\pi_i \widehat{Z}$ by \cite[Theorem 3.14]{quickremarks}. The hypothesis holds for $X$ and $Y$ by assumption, and also for the product $X \times Y$ since a product of good groups is a good group. We conclude that the map $\widehat{X \times Y} \to \widehat{X} \times \widehat{Y}$ is an isomorphism on homotopy groups. 

In order to finish this proof it suffices to prove that weak equivalences in $\widehat{\mathbf{S}}$ are detected on homotopy groups. This is claimed without proof in the last paragraph of the second section of \cite{quickremarks}, and thus we give a quick proof here. In \cite[Theorem 7.4.7]{barneapro}, it is shown that there is a Quillen equivalence
\[\Psi:L\mathrm{Pro}(\mathbf{S})\leftrightarrows\widehat{\mathbf{S}}:\Phi\]
where $L\mathrm{Pro}(\mathbf{S})$ is a certain Bousfield localization of Isaksen's model category on $\mathrm{Pro}(\mathbf{S})$ (defined in \cite{isaksenprospaces}). Since weak equivalences in $\mathrm{Pro}(\mathbf{S})$ are detected on homotopy groups, it follows that the same is true in $\widehat{\mathbf{S}}$.
 \end{proof}

\begin{remark}
The proposition also holds if $X$ and $Y$ have finitely many path components. However, it does not hold when the set of components is infinite. As a counterexample, let us take $X$ and $Y$ to be the discrete space $\mathbb{N}$, we claim that the map 
\[\widehat{\mathbb{N}\times\mathbb{N}}\to\widehat{\mathbb{N}}\times\widehat{\mathbb{N}}\]
is not an isomorphism. Indeed, let us consider the map $f:\mathbb{N}\times\mathbb{N}\to \{0,1\}$ which sends $(x,y)$ to $1$ if $x>y$ and $0$ otherwise. Then this map extends to a map $\widehat{\mathbb{N}\times\mathbb{N}}$ by universal property of the profinite completion. However, it is an easy exercise to check that, if $P$ and $Q$ are two equivalence relations on $\mathbb{N}$ that are such that the quotients $\mathbb{N}/P$ and $\mathbb{N}/Q$ are finite, then there cannot exist a map $g$ that makes the triangle 
\[
\xymatrix{
\mathbb{N}\times\mathbb{N}\ar[r]\ar[rd]_f& (\mathbb{N}/P)\times(\mathbb{N}/Q)\ar[d]^g\\
                                       &\{0,1\}
}
\] 
commute. In fact, one can even show that if there exists such a $g$, then $P$ and $Q$ must both be the finest equivalence relation on $\mathbb{N}$ (i.e. the one for which the equivalence classes are singletons).
\end{remark}
\section{$\infty$-operads}\label{sec:operads}

The profinite completion of a product of spaces is not in general isomorphic to the product of the profinite completions. However, in favorable cases, as in Proposition \ref{prop:prodcompl}, the comparison map is a weak equivalence. For this reason -- and under the assumptions of Proposition \ref{prop:prodcompl} -- the profinite completion of an operad does not yield an operad but rather an $\infty$-operad. We explain this in detail below, after giving a meaning to the term $\infty$-operad which suits our purpose.

We use the dendroidal category $\Omega$ from \cite{mw}. Objects of $\Omega$ are finite rooted trees.  Each such tree $T\in\Omega$ generates a colored operad $\Omega(T)$ which has the set of edges of $T$ for colors and operations are generated by the vertices of $T$. A morphism in $\Omega$ from $S$ to $T$ is defined as an operad map from $\Omega(S)$ to $\Omega(T)$. (For more details, see \cite{mw}.)

Some objects in $\Omega$ are given special mention and notation: the tree with no vertices (i.e. consisting of a single edge, the root) is denoted $\eta$; the tree with a single vertex and $n+1$ edges is called the $n^{th}$ corolla and is denoted by $C_{n}$. For a tree $T$ and vertex $v$ with $|v|$-input edges, there is an inclusion $C_{|v|} \hookrightarrow T$ which selects the edges connected to $v$; this gives rise to a map $X_T \to X_{C_{|v|}}$ for $X$ a contravariant functor on $\Omega$. A \emph{dendroidal object} in $\mathbf{C}$ is a contravariant functor $X$ from $\Omega$ to $\mathbf{C}$. The category of dendroidal objects is denoted $\mathbf{dC}$. 

\begin{definition}\label{defn:inftyoperad} Let $\mathbf{C}$ be a model category. An \emph{$\infty$-operad} in $\mathbf{C}$ is a dendroidal object $X$ such that the map $X_{\eta} \to *$ is a weak equivalence and, for every tree $T$, the Segal map
\[
X_T \to \prod_{v \in T} (X_{C_{|v|}})_f
\]
induced by the inclusion of corollas in $T$, is a weak equivalence in $\mathbf{C}$. Here, the product runs over all vertices of $T$, $|v|$ denotes the set of inputs at $v$, and the subscript $f$ denotes a fibrant replacement in $\mathbf{C}$. 
 \end{definition}
 
A couple of comments on the definition above: 
The first condition should be interpreted as the requirement that the dendroidal object $X$ has a single color.  The role of the fibrant replacements is to guarantee that the product is the homotopy product. In many cases, like in the category of spaces or groupoids, every finite product is a homotopy product hence such fibrant replacements are not needed. 

\medskip
Given  a model category $\mathbf{C}$, the category of dendroidal objects $\mathbf{d}\mathbf{C}$ admits a model category structure in which a map $X \to Y$ is a weak equivalence if $X_T \to Y_T$ is a weak equivalence in $\mathbf{C}$ for every tree $T$. We say such a weak equivalence is given levelwise.

\begin{definition}
The relative category of infinity operads $\Op_{\infty}(\mathbf{C})$ is the full relative subcategory of $\mathbf{d}\mathbf{C}$ spanned by $\infty$-operads (Definition \ref{defn:inftyoperad}) and the levelwise weak equivalences between them.
\end{definition}

By definition, given two $\infty$-operads $X$ and $Y$, the derived mapping space in $\Op_{\infty}$ is simply the derived mapping space $\Rmap(X,Y)$ computed in $\mathbf{dC}$.

\medskip
When $\mathbf{C}$ is the category of spaces or groupoids, an $\infty$-operad will be a fibrant object in the left Bousfield localization of $\mathbf{dC}$ at the Segal maps
$$\coprod_{v \in T} \Omega(-,C_{|v|}) \to \Omega(-,T)$$
for every tree $T$, and at the map $\varnothing \to \Omega(-,\eta)$. We denote this model structure by $L_S(\mathbf{dC})$. (See also \cite[Around Proposition 5.5]{cm2} and \cite[Proposition 4.3]{bh} for the monochromatic case.) If $\mathbf{C}$ is the category of profinite spaces, there are technical challenges to performing a similar left Bousfield localization. This stems from the fact that the model structure on profinite spaces is not cofibrantly generated but rather fibrantly generated. The definition of $\Op_{\infty}(\mathbf{C})$ above allows us to circumvent this issue, while still being robust enough for our purposes. 

\subsection{The nerve of an operad}
For a monochromatic operad $\mathsf{P}$ in $\mathbf{C}$, one can associate a dendroidal object $N\mathsf{P}$, called the \emph{nerve} of $\mathsf{P}$, by declaring
\[
(N \mathsf{P})_T := \prod_{v \in T} \mathsf{P}(|v|)
\]
for every tree $T$. Note that the value of $N\mathsf{P}$ at $\eta$ is a point. This is an $\infty$-operad if $\mathsf{P}$ is levelwise fibrant or if in $\mathbf{C}$ finite products and finite homotopy products agree. 

\medskip
The following proposition and the subsequent theorem are essentially due to Cisinski-Moerdijk, slightly adjusted to our context.
\begin{prop}\label{prop:nerve-is-ff}
Let $\mathbf{C}$ be the category of spaces or groupoids.
The nerve $N$ is homotopically fully faithful, i.e. for every pair of (monochromatic) operads $\mathsf{P}$ and $Q$ in $\mathbf{C}$, the map
\[
\Rmap(\mathsf{P},\mathsf{Q}) \to \Rmap(N \mathsf{P}, N \mathsf{Q})
\]
is a weak equivalence of spaces. Here the derived mapping spaces are computed in $\Op(\mathbf{C})$ and $\Op_{\infty}(\mathbf{C})$, respectively.
\end{prop}

The following technical step will be used in the proof of Proposition \ref{prop:nerve-is-ff}.

\begin{lemma}\label{reduction_lemma}Let $\mathbf{C}$ denote the category of spaces, and let $Z:=\Omega(-,C_{n}) \times K$ be an object of $\mathbf{dC}$ with $K$ a space and $C_n$ a corolla. Let $F(K)$ denote the free operad on the (non-symmetric) sequence with $K$ in degree $n$ and empty otherwise. Then the map 
\[
Z \to NF(K)
\] 
is a weak equivalence in $L_S(\mathbf{dC})$.
\end{lemma} 

\begin{proof} 
For a dendroidal space $Z$, the \emph{reduction} $Z_*$ is the dendroidal space given by the pushout
\begin{equation}\label{reduction}
\xymatrix{
\Omega(-,\eta) \times Z_\eta \ar[r]\ar[d]& Z \ar[d]^-{}\\
\Omega(-,\eta) \ar[r]& Z_{*}
}
\end{equation} 
where the top horizontal map is the adjoint to the identity. The map in the statement of the lemma factors as
\begin{equation}\label{composite}
Z \to Z_* \to NF(K) \; .
\end{equation}
According to \cite[Proposition 4.4]{bh}, the right-hand map is a weak equivalence in a Segal-type model structure obtained from $\mathbf{dC}_*$, the category of dendroidal objects $X$ which have $X_\eta$ a point. The inclusion of that model structure into $L_S(\mathbf{dC})$ preserves all weak equivalences, hence the right-hand map is also a weak equivalence in $L_S(\mathbf{dC})$.

It remains to show that the left-hand map of (\ref{composite}) is a weak equivalence. The top horizontal map of the square (\ref{reduction}) is a Reedy cofibration of dendroidal spaces, therefore the square is a homotopy pushout. Moreover, the left-hand map in that square is a weak equivalence in $L_S(\mathbf{dC})$, so the result follows.
\end{proof}

\begin{proof}[Proof of Proposition~\ref{prop:nerve-is-ff}] We take $\mathbf{C}$ to be the category of spaces, and deduce the result for groupoids at the end of the proof.

The nerve functor preserves all weak equivalences and thus it induces a map between derived mapping spaces as in the statement.

Consider the case when $\mathsf{P}$ is a free operad, generated by a (non-symmetric) sequence $X = \{X(n)\}_{n\geq 0}$. In this case, the composite
\begin{equation}\label{maincomposite}
\Rmap(\mathsf{P},\mathsf{Q}) \to \Rmap(N \mathsf{P}, N \mathsf{Q}) \to \Rmap(\coprod_{n \geq 0} \Omega(-,C_n) \times X(n), N \mathsf{Q})
\end{equation}
is a weak equivalence. This is because both source and target derived mapping spaces are equivalent to the non-derived mapping spaces and $\map(\mathsf{P},\mathsf{Q})$ is isomorphic to $\map(\coprod_{n \geq 0} \Omega(-,C_n) \times X(n), N \mathsf{Q})$. The right-hand map is a weak equivalence by lemma~\ref{reduction_lemma}. By two-out-of-three, we can conclude that the left-hand map is also a weak equivalence.

To extend this result to a general operad $\mathsf{P}$, we resolve $\mathsf{P}$ by free operads. In other words, we form a simplicial object whose operad of $n$-simplices is given by $$F^{n+1}(\mathsf{P}) := \underbrace{F \circ \dots \circ F}_{n+1}(\mathsf{P}) \; ,$$ where $F$ denotes the free operad construction. This simplicial object $F^{\bullet+1}\mathsf{P}$ is augmented via the map $F(\mathsf{P}) \to \mathsf{P}$ given by operadic composition. It is also the case that $F^{\bullet+1}\mathsf{P}$ is homotopy equivalent to $\mathsf{P}$ as simplicial objects in simplicial operads, and that $F^{\bullet+1}\mathsf{P}$ is Reedy cofibrant. It follows that the maps
\begin{equation}\label{eq:freeP}
\hocolimsub{\Delta} F^{\bullet + 1}(\mathsf{P}) \to |F^{\bullet + 1}(\mathsf{P})| \to \mathsf{P}
\end{equation}
are weak equivalences. 

Now, we apply the nerve functor to this simplicial object, and investigate the map
\[
\hocolimsub{\Delta} NF^{\bullet + 1}(\mathsf{P}) \to N\mathsf{P} \; .
\]
First of all, this is a map of $\infty$-operads: homotopy colimit over $\Delta$ commutes with homotopy products, and so it preserves the Segal condition. Therefore, to verify that the map is a weak equivalence it is sufficient to show that it is a weak equivalence on corollas. But the latter condition is precisely the condition that the composite map (\ref{eq:freeP}) is a weak equivalence.
We have thus reduced the general case to the case of free operads, which we have already dealt with.

We now turn to the case of groupoids. The nerve functor $B : \G \to \S$ preserves products and is homotopically fully faithful. Therefore, by levelwise application, $B$ induces a functor from operads in groupoids to operads in spaces and from $\infty$-operads in groupoids to $\infty$-operads in spaces. For operads in groupoids $\mathsf{P}$ and $\mathsf{Q}$, we obtain a commutative square
\[
\xymatrix{
\Rmap(\mathsf{P},\mathsf{Q})\ar[r]\ar[d]&\Rmap(N\mathsf{P},N\mathsf{Q})\ar[d]\\
\Rmap(B\mathsf{P},B\mathsf{Q})\ar[r]&\Rmap(BN\mathsf{P},BN\mathsf{Q})
}
\]
We have shown above that the lower horizontal map is a weak equivalence (together with the observation that $BN \cong NB$). The vertical maps are weak equivalences because $B$ is homotopically fully faithful. This completes the proof.
\end{proof}

\begin{thm}[Cisinski-Moerdijk]\label{thm:CM} Let $\mathbf{C}$ denote the category of spaces or groupoids. The nerve functor $N$ is a right Quillen equivalence from the model structure on monochromatic operads in $\mathbf{C}$ to $L_S(\mathbf{dC})$, the localization of the projective model structure on dendroidal objects in $\mathbf{C}$ where the fibrant objects are the $\infty$-operads.
\end{thm}
\begin{proof}
The nerve functor has a left adjoint $L$. The value of $L$ on a representable $\Omega(-,T)$ is the free operad on a sequence $\{X(n)\}$ with $X(n)$ the set of vertices of $T$ with $n$ inputs. This prescription uniquely defines $L$. The pair $(L,N)$ is Quillen since $N$ preserves fibrations and weak equivalences. Moreover, $N$ detects weak equivalences, i.e. $N$ is homotopy conservative. Therefore, to show that the pair is a Quillen equivalence, it is enough to show that the derived unit map is a weak equivalence. This follows from Proposition \ref{prop:nerve-is-ff}.
\end{proof}

\section{Profinite completion of operads}\label{sec:pro-operads}

The adjunctions (\ref{adjunction_profinite_group}) and (\ref{adjunction profinite completion}) relating groupoids and profinite groupoids and the space version, give rise to simplicial Quillen adjunctions
\[
\mathbf{d}\S \leftrightarrows \mathbf{d}\widehat{\S} \quad \mbox{ and } \quad \mathbf{d}\G \leftrightarrows \mathbf{d}\widehat{\G}
\]
where all the categories are equipped with model structures where weak equivalences are given levelwise.

Given an $\infty$-operad $X$ in spaces or groupoids, write $\widehat{X}$ for the dendroidal object obtained by applying profinite completion levelwise. In general, $\widehat{X}$ is not an $\infty$-operad. If it is, then we have a weak equivalence
\[
\Rmap(\widehat{X}, Y) \simeq \Rmap(X, |Y_f|)
\]
natural in $Y \in \Op_{\infty}(\widehat{\S})$, where $|Y_f|$ is the $\infty$-operad in spaces whose value at a tree $T$ is $|(Y_T)_f|$, and $f$ is a fibrant replacement functor in $\widehat{\S}$. In these circumstances, we call $\widehat{X}$ the \emph{profinite completion} of the $\infty$-operad $X$.

In general, it is reasonable to define the profinite completion of $X$ as the $\infty$-operad characterized by the formula above. In other words, it is the $\infty$-operad in profinite spaces which corepresents the functor 
\[ Y \mapsto \Rmap(X, |Y_f|) \]
for $Y \in \Op_{\infty}(\widehat{\S})$. At any rate, in the cases that we are interested in, the levelwise profinite completion always produces an $\infty$-operad, using the following observation.

\begin{prop}\label{prop:Nisop}
Let $\mathsf{P}$ be an operad in spaces such that each $\mathsf{P}(n)$ has finitely many components and its homotopy groups are good. Then $(N\mathsf{P})^{\wedge}$ is an $\infty$-operad in profinite spaces.
\end{prop}
\begin{proof}
This is an immediate consequence of Proposition \ref{prop:prodcompl}.
\end{proof}

\begin{remark}\label{rem:Nisopgp}
For operads in groupoids, the situation is nicer. Indeed, if $\mathsf{P}$ is an operad in groupoids in which each groupoid $\mathsf{P}(n)$ has finitely many objects, then $(N\mathsf{P})^{\wedge}$ is a strict operad in profinite groupoids by Remark \ref{rem:prod gpoids}. This applies for instance to the operads $\pab$ and $\parb$ defined in the next section.
\end{remark}

\section{Braids and Ribbon Braids}\label{ribbon_braid}

The braid group on $n$ strands, hereafter denoted $\br(n),$ is the fundamental group of the space of unordered configurations of $n$ points in the complex plane. This group has a preferred presentation with generators $\{\boldsymbol{\beta}_i\}_{1\leq i\leq n-1}$ which are subject to the so-called Artin relations:
\begin{itemize}
\item $\boldsymbol{\beta}_i\boldsymbol{\beta}_j=\boldsymbol{\beta}_j\boldsymbol{\beta}_i$ if $|i-j|\geq2$.
\item $\boldsymbol{\beta}_i\boldsymbol{\beta}_{i+1}\boldsymbol{\beta}_i=\boldsymbol{\beta}_{i+1}\boldsymbol{\beta}_i\boldsymbol{\beta}_{i+1}$.
\end{itemize}
The braid and symmetric groups fit in a short exact sequence of groups
\[1\to\pb(n)\to\br(n)\to \Sigma_n\to 1\]
where $\pb(n)$ is the pure braid group on $n$ strands. The pure braid group is also the fundamental group of the space of ordered configurations of points in the plane. In terms of generators, the map $\br(n)\to\Sigma_n$ sends the elementary braid $\boldsymbol{\beta}_i$ to the permutation $(i,i+1)$.  

In this paper, we are also concerned with the ribbon versions of these two groups. The ribbon braid group, denoted $\rb(n)$ is the fundamental group of the space of unordered configurations of $n$ points in the plane, where each point is equipped with a choice of a label in $S^1$. There is an obvious map $\rb(n)\to\br(n)$ that corresponds, at the space level, to forgetting the data of the label. This map is split surjective (a section exists at the space level by giving each point in the configuration a fixed label). There is a presentation of the ribbon braid group $\rb(n)$ that is compatible with the inclusion of $\br(n)$. It has generators $\boldsymbol{\beta}_i$ with $i$ in $\{1,\ldots,n-1\}$ and $\boldsymbol{\tau}_j$ with $j$ in $\{1,\ldots n\}$ subject to the relations
\begin{itemize}
\item $\boldsymbol{\beta}_i\boldsymbol{\tau}_j=\boldsymbol{\tau}_j\boldsymbol{\beta}_i$ for $j\notin \{i,i+1\}$.
\item $\boldsymbol{\beta}_i\boldsymbol{\tau}_{i+1}=\boldsymbol{\tau}_i\boldsymbol{\beta}_i$.
\item $\boldsymbol{\tau}_i\boldsymbol{\tau}_j=\boldsymbol{\tau}_j\boldsymbol{\tau}_i$ if $i\neq j$.
\end{itemize}
as well as the Artin relations. Another way to think of this group is to let $\br(n)$ act on the left on $\mathbb{Z}^n$ by the composite
\[\br(n)\xrightarrow{}\Sigma_n\to \mathrm{GL}_n(\mathbb{Z})\]
where the second map is the map that sends a permutation to its permutation matrix. The reader can easily check from the above presentation that there is an isomorphism $\rb(n)\cong \br(n) \ltimes \mathbb{Z}^n$.

There is also a map $\rb(n)\to\Sigma_n$ that is given in terms of generators by sending $\boldsymbol{\beta}_i$ to $(i,i+1)$ and $\boldsymbol{\tau}_j$ to the identity. This map is surjective and its kernel is the pure ribbon braid group on $n$ strands denoted $\prb(n)$. This group is also the fundamental group of the space of ordered configurations of $n$ points in the plane each equipped with a choice of a label in $S^1$. This space splits as a product of the space of ordered configurations of $n$ points in the plane with the space $(S^1)^n$. It follows that $\prb(n)$ splits as $\pb(n)\times\mathbb{Z}^n$.

\subsection{Colored (ribbon) braid operad} 
In this section, we describe two operads in groupoids, $\cob$ and $\corb$, which are central to the paper. They are models for the operad of little $2$-disks and its framed version (in the variants without $0$-arity operations). 

We first recall the definition of the non-unital associative operad below. 

\begin{definition}\label{def_symmgrp_operad} The operad $\Sigma$ is an operad in sets whose arity zero term is the empty set and whose arity $n$ term (for $n$ positive) is the symmetric group $\Sigma_n$. Operadic composition $\circ_{k}:\Sigma_{m}\times\Sigma_{n}\rightarrow \Sigma_{m+n-1}$ consists of placing $\sigma\in\Sigma_{n}$ in position $k$ of $\tau\in\Sigma_m$ and reindexing. For example $(132)\circ_{2}(12)=(1423)$. 
\end{definition}

\begin{definition}\label{cob}  The operad of colored braids $\cob=\{\cob(n)\}_{n\geq0}$ consists of a collection of groupoids $\cob(n)$ defined as follows.
\begin{itemize}
\item $\cob(0)$ is the empty groupoid.
\item For $n>0$, the set of objects $\ob(\cob(n))$ is $\Sigma_{n}$ \,.
\item A morphism in $\cob(n)$ from $p$ to $q$ is a braid $\alpha\in\br(n)$ whose associated permutation is $qp^{-1}$.
\end{itemize}

The categorical composition in $\cob(n)$, 
$$\Hom_{\cob(n)}(p,q)\times\Hom_{\cob(n)}(q,t)\rightarrow \Hom_{\cob(n)}(p,t)$$
is given by the concatenation operation of braids, inherited from the braid group. We write $a \cdot b$ for the \emph{categorical composition} of $a$ and $b$.

The operadic composition operation
$$\circ_{k}:\cob(m) \times \cob(n)\longrightarrow \cob(m+n-1)$$
is defined as follows. On objects, it is given by operadic composition of permutations, as in the associative operad. On morphisms, it corresponds to replacing a chosen strand by a braid; given morphisms $\alpha$ in $\cob(m)$ and $\beta$ in $\cob(n)$, the braid $\alpha \circ_k \beta$ is obtained by replacing the $k^{th}$ strand in $\alpha$ by the braid $\beta$ as in the picture below:

\begin{center}
\includegraphics[scale=0.3]{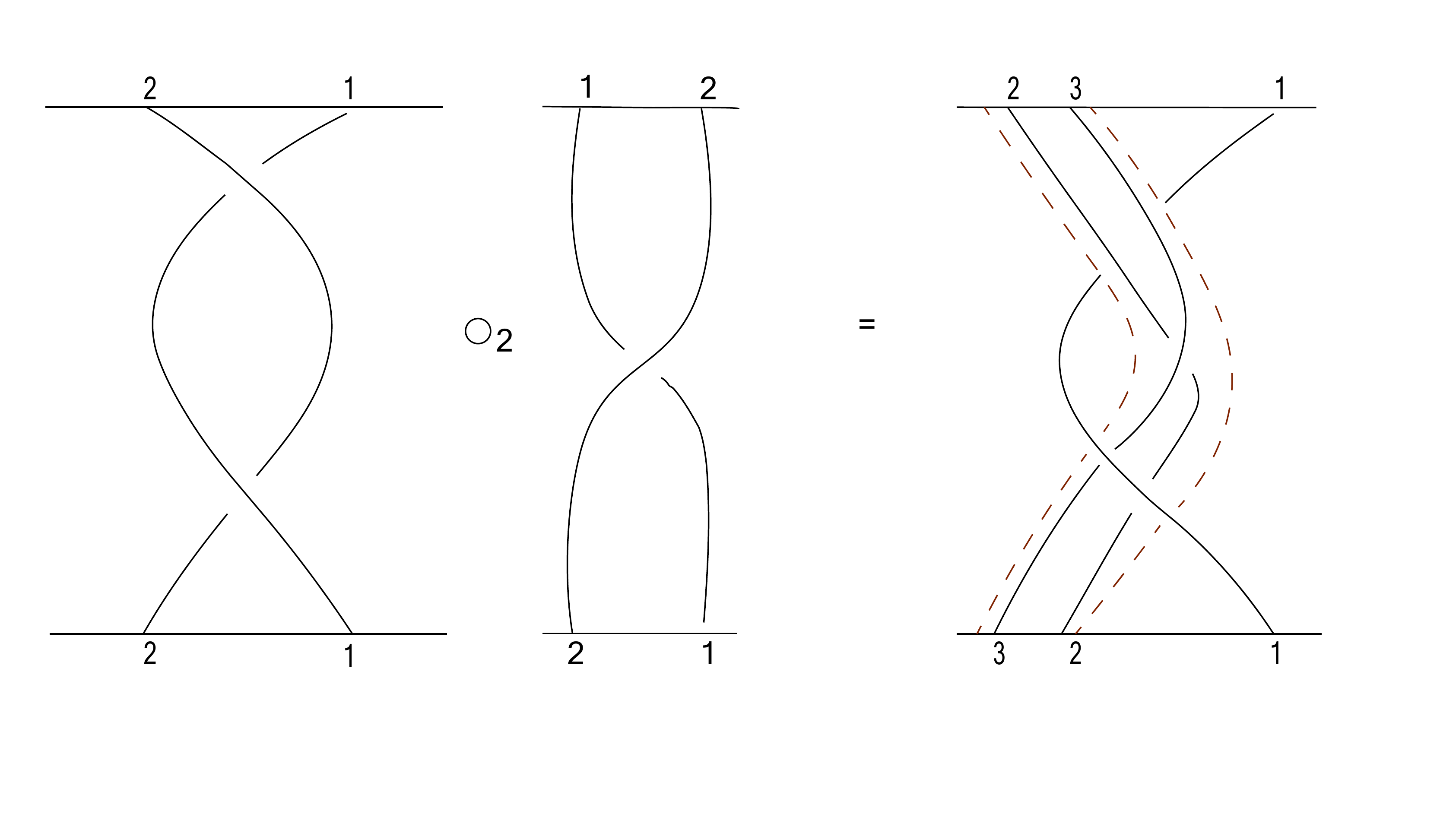}
\end{center} 
\end{definition}

\begin{remark}
Any morphism in $\cob(n)$, i.e. a braid, has an expression as a categorical composition of elementary braids $\boldsymbol{\beta}_i$. Moreover, each elementary braid can be expressed as an operadic composition of an identity morphism (a trivial braid) and a morphism $\boldsymbol{\beta}$ in $\cob(2)$, the non-trivial braid on two strands pictured below. 

\begin{center}
\includegraphics[scale=0.3]{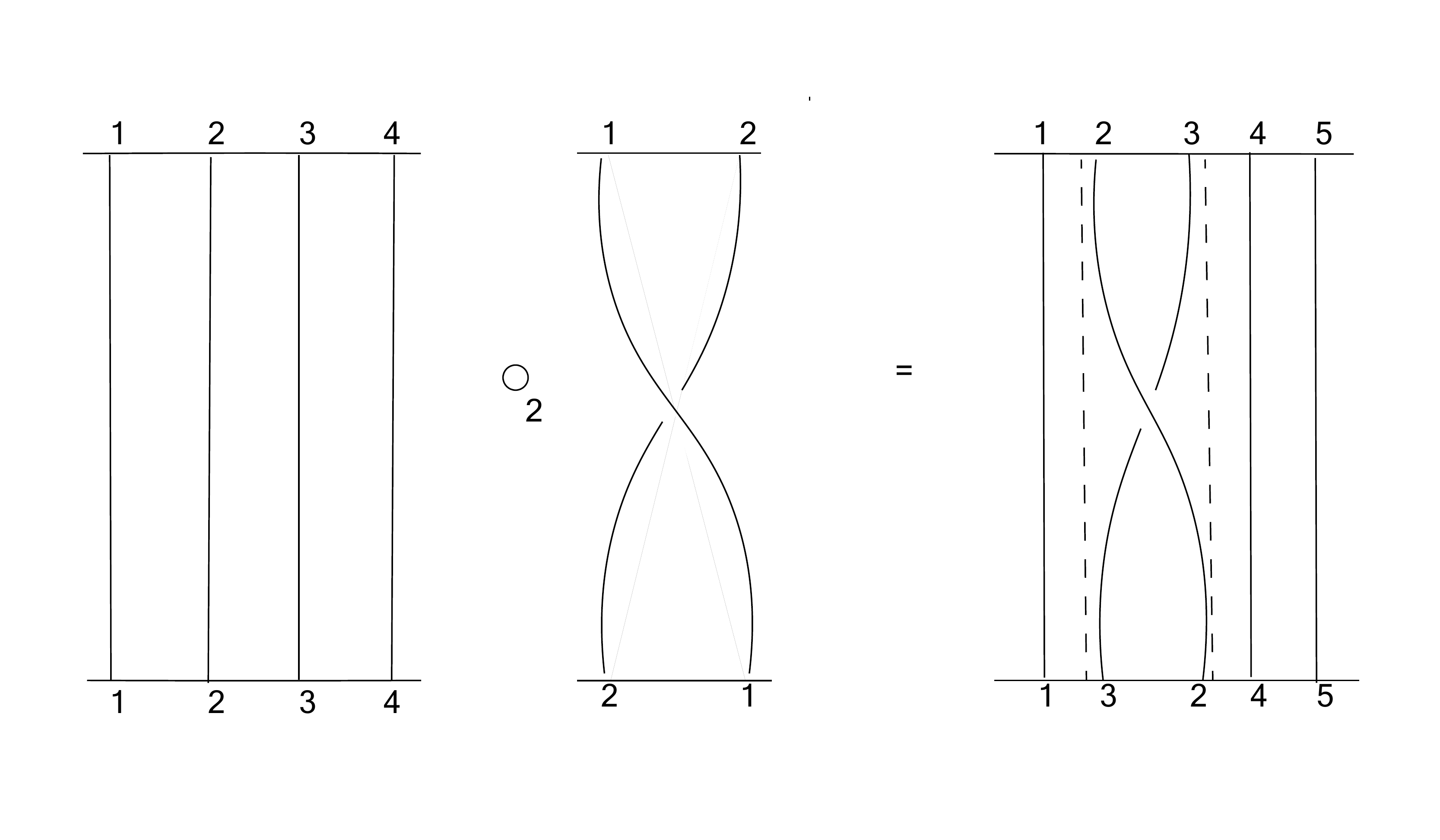}
\end{center} 

 Therefore any morphism in $\cob(n)$ can be expressed as a categorical composition of morphisms obtained as operadic compositions of identities and $\boldsymbol{\beta}$.
\end{remark} 

The operad $\cob$ has a ribbon version $\corb$ that we now define.

\begin{definition}\cite[Example 1.2.9]{wahlthesis}\label{corb}
The groupoid of colored ribbon braids on $n$ strands $\corb(n)$ is the groupoid defined as follows.

\begin{itemize} 
\item $\corb(0)$ is the empty groupoid.
\item For $n>0$, the set of objects $\ob(\corb(n))$ is $\Sigma_{n}$ \,.
\item A morphism in $\corb(n)$ from $p$ to $q$ is a pair $(\gamma, [x_1, \dots, x_n])$ where $\gamma \in \Hom_{\cob(n)}(p,q)$ and $x_i \in \mathbb{Z}$. We think of such a morphism as a braid equipped with the additional data of a twisting number $x_i$ for each strand $i$.
\end{itemize} 

Composition of morphisms in $\corb(n)$ is given by composing the morphisms in $\cob(n)$ and adding the twisting numbers. The identity element is the trivial braid with no twists on each strand.

The sequence of groupoids $\corb=\{\corb(n)\}_{n\geq0}$ forms an operad. On objects, it is the associative operad, as for $\cob$. To define the operadic composition on morphisms, we first introduce some notation: for a non-negative integer $m$, we write $\boldsymbol{R}_m$ for the element $(\boldsymbol{\beta}_1 \dots \boldsymbol{\beta}_{m-1})^m$ in $\pb(m)$. Given morphisms $(\gamma, [x_1, \dots, x_n])$ and $(\alpha, [y_1, \dots, y_m])$ in $\corb(n)$ and $\corb(m)$, respectively, their operadic composition (at entry $k$) is the morphism $$(\omega, [x_1, \dots, x_{k-1}, x_{k} + y_1, \dots, x_k + y_m, \dots, y_m])$$ in $\corb(n+m-1)$ with
\[
\omega := \gamma \circ_k ((\boldsymbol{R}_m)^{x_k} \cdot \alpha) 
\]
where $(\boldsymbol{R}_m)^{x_k}$ is the $x_k$-fold categorical composition of $\boldsymbol{R}_m$ considered as an automorphism in $\cob(m)$, and $\circ_k$ is the operadic composition product in $\cob$. Since $\boldsymbol{R}_m$ is an element of the center of $\pb(m)$, the operation just defined is compatible with the groupoid structure. The operadic identity is given by the trivial braid with one strand and no twists. 

Below is a picture of a special case of the operadic composition in $\corb$, corresponding to $(id,[1]) \circ (id, [0,0])$. (When it is non-zero, we draw the twisting number of a strand in a grey box over that strand.)

\begin{center}
\includegraphics[scale=0.3]{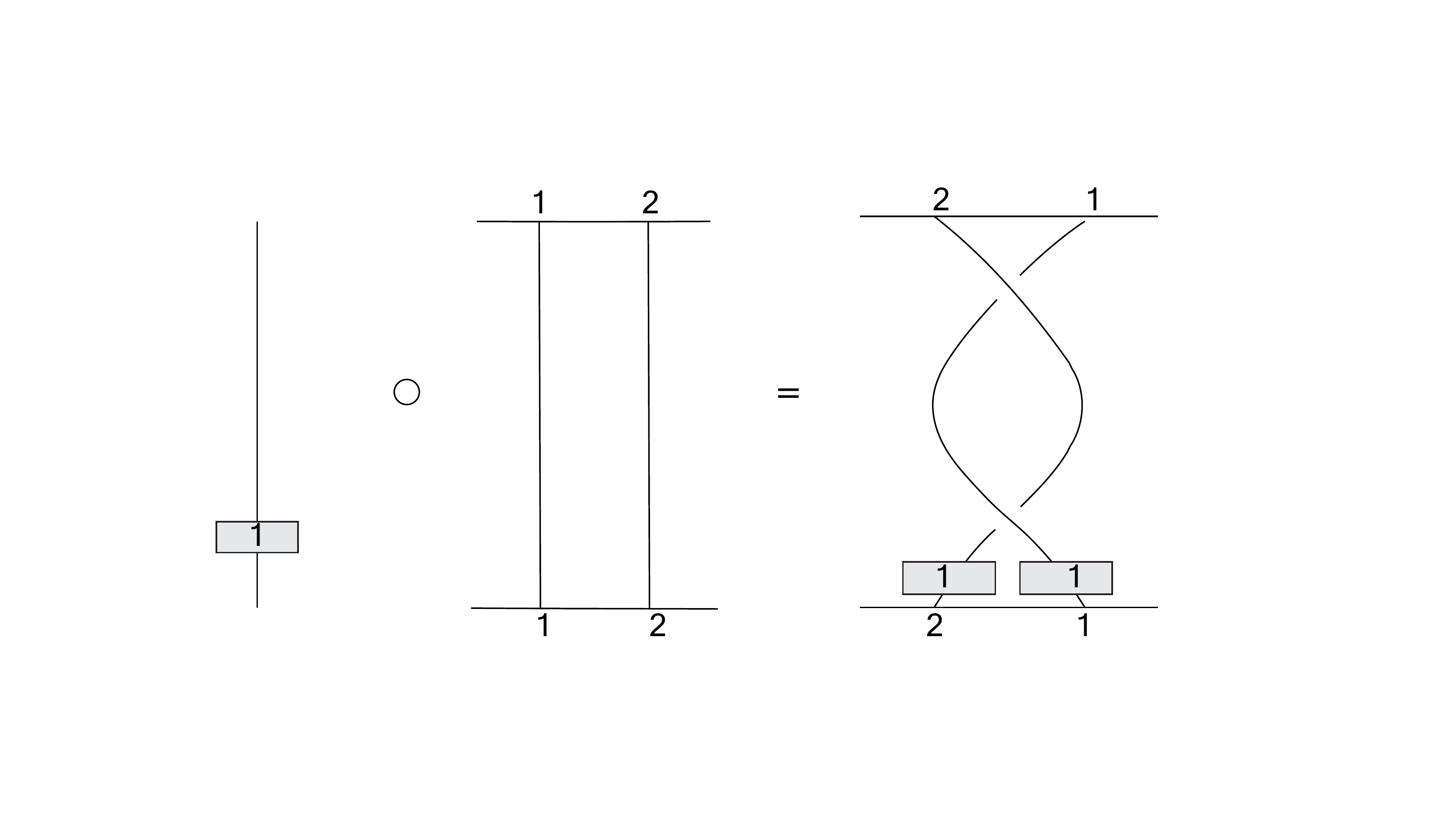}
\end{center} 

Another example of an operadic composition in $\corb$ is pictured below.

\begin{center}
\includegraphics[scale=0.3]{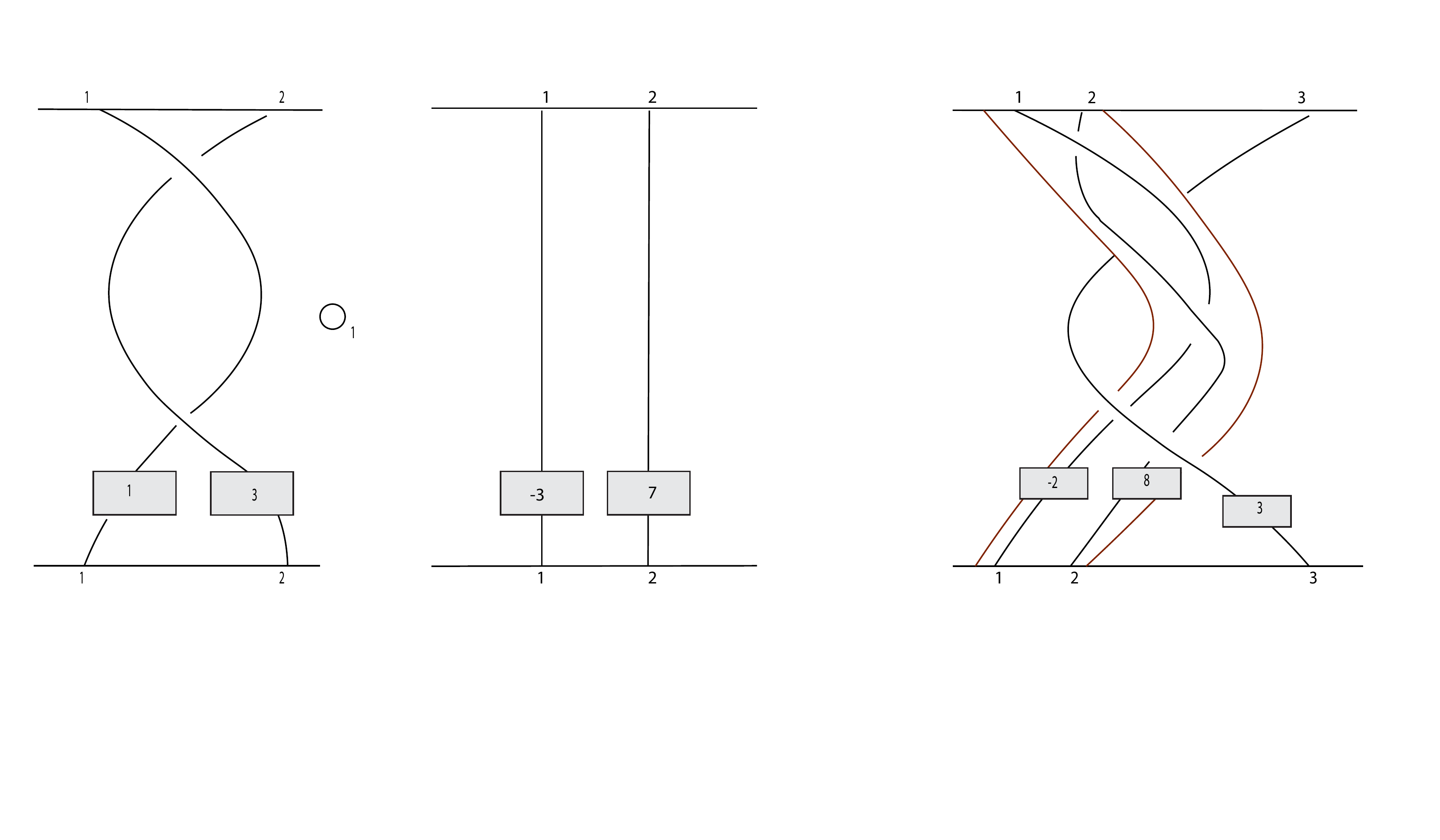}
\end{center}

\end{definition}

The operad $\corb$ is a model for the genus zero surface operad and, equivalently, the framed little $2$-disks operad whose definitions we now recall. 

Let $F_{0,n+1}$ denote an oriented surface of genus zero and with $n+1$ boundary components. We choose one of these boundary components and call it marked and we say the other $n$ components are free. We also assume each boundary component $\partial_i$ comes equipped with a collar, i.e. a neat embedding of $[0, \epsilon)\times  S^1$ in a neighborhood of $\partial_i$.  The mapping class group $\Gamma_{0}^{n+1}$ is the group of isotopy classes of orientation preserving diffeomorphisms of $F_{0,n+1}$ which fix collars pointwise. By gluing the boundary components of surfaces one obtains composition maps, but these are not strictly associative nor unital. To fix the issue, we recall a construction of Tillmann~\cite[Construction 2.2]{tillmann}, and a later improvement by Wahl \cite[3.1]{wahlinfiniteloop}, which replaces the mapping class group $\Gamma_{0}^{n+1}$ with an equivalent connected groupoid. 

We begin by defining a groupoid $\mathcal{E}_{n+1}$ whose objects are surfaces with $n$ free boundary components having a particular decomposition into standard pairs of pants $P$ and standard disks $D$. By a standard pair of pants we mean a fixed pair of pants with (variable) collars at each boundary component and an assigned ordering of the free boundary components. A morphism in $\mathcal{E}_{n+1}$ is an isotopy class of diffeomorphisms that preserves the boundary collars and their ordering. By gluing the marked boundary component of a surface to the $i^{th}$ free boundary component of another surface, one obtains composition maps $\circ_i : \mathcal{E}_{n+1} \times \mathcal{E}_{m+1} \to \mathcal{E}_{n+m}$. These are associative, but still not unital. Moreover, there are now too many objects in $\mathcal{E}_{n+1}$, so Tillmann identifies two such surfaces if one can be obtained from the other by replacing a subsurface of the form $P \circ_1 P$ by a subsurface of the form $P \circ_2 P$. To have a strict unit, Tillmann then introduces a second identification by declaring two surfaces equal if one can be obtained from the other by replacing a subsurface of the form $P \circ_1 D$ or $P \circ_2 D$ by a circle. After making these two identifications, any object has a unique representative as a surface having no subsurfaces of the form $P \circ_2 P$, $P \circ_1 D$ or $P \circ_2 D$. For $n \neq 1$, let $\mathcal{S}_n$ denote the full subgroupoid of $\mathcal{E}_n$ spanned by these special surfaces. As explained by Wahl, it is a consequence of the Alexander trick that there is a canonical way to, given a morphism in $\mathcal{E}_n$, produce a morphism in $\mathcal{S}_n$ such that the resulting maps
\[
\circ_i : \mathcal{S}_{n+1} \times \mathcal{S}_{m+1} \to \mathcal{E}_{n+1} \to \mathcal{S}_{n+m}
\]
are associative. Now we can define the surface operad.

\begin{definition} \label{def:surface_operad}\cite[2.3]{tillmann}\cite[3.1]{wahlinfiniteloop} The surface groupoids $\mathcal{S}_{+}(n)$ are defined by
\begin{itemize} 
\item $\mathcal{S}_{+}(0)$ is the trivial groupoid, whose only object is the standard disk $D$,
\item $\mathcal{S}_{+}(1)$ is a groupoid with one object $S^1$ and with $\mathbb{Z}$ as morphisms (thought of as the Dehn twists around that circle) and
\item  for $n\geq 2$, $\mathcal{S}_{+}(n)$ is the groupoid $\mathcal{S}_n$ defined above. Namely, an object in $\mathcal{S}_{+}(n)$ is a surface having no subsurfaces of the form $P \circ_2 P$, $P \circ_1 D$ or $P \circ_2 D$, together with an ordering of the boundary components and a choice of collar around each. Morphisms are isotopy classes of orientation preserving diffeomorphisms that fix the collars and their ordering.
\end{itemize}
The groupoids $\mathcal{S}_{+}(n)$ assemble into an operad. The operadic composition maps $\circ_i$ in $\mathcal{S}_+$ are induced from $\mathcal{E}$ as explained above. When no arity one operations are involved, this corresponds to the gluing of surfaces at the object-set level. Composition with arity zero operations is essentially filling in a boundary component of a surface. Composition with arity one operations maps is more subtle. When $n > 1$, the composition map \[\circ_{i}:\mathcal{S}_{+}(n)\times \mathcal{S}_{+}(1) \to \mathcal{S}_{+}(n)\] sends a surface to itself, but changes the diffeomorphism on the collar of the relevant free boundary component by a Dehn twist. When $n = 1$, the composition corresponds to addition of integers.

The space $B\mathcal{S}_{+}(n)$ is homotopy equivalent to $B\Gamma_{0}^{n+1}$ and we define the unital genus zero surface operad $\mathcal{M}_{+}$ to be the operad $$\mathcal{M}_{+}(n)=B\mathcal{S}_{+}(n)\simeq B\Gamma_{0}^{n+1}\; .$$ 

The nonunital variant of $\mathcal{S}_{+}$ is defined as 
$$\mathcal{S}(n)=\begin{cases}\emptyset \ \textnormal{for} \ n = 0 \\ \mathcal{S}_+(n) \ \textnormal{for all} \ n>0 \end{cases}$$
and the genus zero surface operad $\mathcal{M}$ as $$\mathcal{M}(n)=B\mathcal{S}(n)\simeq B\Gamma_{0}^{n+1} \; .$$

\end{definition} 

\begin{remark}
There exists a canonical isomorphism between the groupoids $\mathcal{S}(n)$ and $\corb(n)$ which extends to an isomorphism of operads. We alternatively could have defined the genus zero surface operad $\mathcal{M}$ as $B\corb$.
\end{remark}

\begin{definition}
The framed little $2$-disks operad $\mathsf{FD}$ is an operad given in positive arity $n$ by the space of all smooth, orientation-preserving embeddings of the disjoint union of $n$ disks into a single disk (see, for example, \cite[pg. 20]{GetzlerFD} ). We define $\mathsf{FD}(0)$ to be the empty space. 
\end{definition}

By remembering where the center of each disk goes and the value of the derivative at each of those centers, the space $\mathsf{FD}(n)$ is homotopy equivalent to the space of ordered configurations of $n$ points in the disk each equipped with a label in $S^1$. This space is itself homotopy equivalent to $B\prb(n)$. 

\begin{prop}\cite[Prop 1.3.14]{wahlthesis}\label{wahlprop}
The classifying space of the operad $\corb$ is weakly equivalent to the operad $\mathsf{FD}$. 
\end{prop}

\subsection{Parenthesized Ribbons} 

The goal of this section is to give a cofibrant resolution of the operad $\corb$, called the parenthesized ribbon braid operad and denoted by $\parb$.

The operad of objects of $\parb$ is the free operad generated by a single operation in arity two. We give a more concrete description below.

\begin{definition} 
Let $S$ be a finite set. We define the set of non-associative monomials of length $n$, $\bM_{n}S$, inductively as \begin{itemize} 
\item  $\bM_0 S =\emptyset$, 
\item $\bM_{1}S = S$ and
\item  $\bM_{n}S = \coprod_{p+q=n} \bM_{p}S \times \bM_{q}S.$ 
\end{itemize} 
\end{definition} 

Alternatively, $\bM_nS$ is the set of rooted, binary (i.e. each vertex has exactly two incoming edges) planar trees with $n$ leaves labelled by elements of $S$. A short-hand notation for elements $\bM_{n}S$ is as parenthesized words in $S$. For example, for $S = \{a, b,c,d\}$, the expression $(a(db))(ba)$ represents an element in $\bM_5S$.

\begin{definition}\label{magma_operad} 
Let $\mathsf{M}=\{\mathsf{M}(n)\}$ be the symmetric sequence where $\mathsf{M}(n)$ is the subset of $\bM_{n}\{1,...,n\}$ consisting of the monomials in $\{1,...,n\}$ where each element of the set occurs exactly once. The symmetric group $\Sigma_n$ acts from the right on $\mathsf{M}(n)$ by permuting the elements of the set $\{1,...,n\}$. The symmetric sequence $\mathsf{M}$ becomes an operad with operadic composition given by replacing letters by monomials (or grafting binary trees). For example, we have the following composition 
\[(1(34))(25) \circ_4 (13)2=(1(3((46)5)))(27)\,.\]
The operad $\mathsf{M}$ is called the magma operad. 
\end{definition} 

There is an obvious operad map
$u : \mathsf{M} \to \Sigma$
which forgets the parenthesization.

\begin{definition}\label{def:PaB} The operad of parenthesized braids $\pab$ is the operad in groupoids defined as follows.
\begin{itemize} 
\item The operad of objects is the magma operad, i.e. $\ob(\pab)=\mathsf{M}$. 
\item  For each $n\geq0$, the morphisms of the groupoid $\pab(n)$ are morphisms in $\cob(n)$, $$ \Hom_{\pab(n)}(p,q)=\Hom_{\cob(n)}(u(p),u(q)).$$
\end{itemize} 

The collection of groupoids $\{\pab(n)\}_{n\geq0}$ forms an operad. On objects, it has the operad structure of $\mathsf{M}$ and on morphisms that of $\cob$.

\end{definition}

There is also a ribbon version of $\pab$.

\begin{definition}\label{def:PaRB} The operad of parenthesized ribbon braids $\parb$ is the operad in groupoids defined as follows. 

\begin{itemize} 

\item The operad of objects is the magma operad, i.e. $\ob(\parb)=\mathsf{M}$. 
\item For each $n\geq0$, the morphisms of the groupoid $\parb(n)$ are morphisms in $\corb(n)$, $$ \Hom_{\parb(n)}(p,q)=\Hom_{\corb(n)}(u(p),u(q)).$$ 
\end{itemize} 

The collection of groupoids $\{\parb(n)\}_{n\geq0}$ forms an operad in groupoids. On objects, it has the operad structure of $\mathsf{M}$ and on morphisms that of $\corb$.

\end{definition} 

Recall that a map $\mathsf{P}\rightarrow \mathsf{Q}$ of operads in groupoids is a weak equivalence if aritywise $\mathsf{P}(n)\rightarrow \mathsf{Q}(n)$ is a weak equivalence of groupoids.

\begin{lemma}\label{lem:parbisfE2}
The forgetful map $\parb \to \corb$ is a weak equivalence of operads in groupoids. Therefore, $\parb$ is also a model for the genus zero surface operad.
\end{lemma}

\begin{proof}
The map is surjective on object sets for each arity $n$, and it is bijective on morphisms by construction.
\end{proof}

\begin{cor}
$\parb$ is a cofibrant replacement of $\corb$. 
\end{cor}
\begin{proof}
From Proposition \ref{prop: free op gp} we know that an operad $\mathsf{P}$ in groupoids is cofibrant if $\ob(\mathsf{P})$ is free. The magma operad $\mathsf{M}$ (Definition~\ref{magma_operad}) is free on one operation of arity $2$, therefore the result follows.
\end{proof}

\section{Operad maps out of $\parb$}\label{maps_parb}

Throughout this section, $\mathsf{P}$ is a fixed operad in groupoids. For $\sigma \in \Sigma_m$ and $x$ an object or morphism in $\mathsf{P}(m)$, we write $\sigma x$ for the action of $\sigma$ on $x$.

\begin{lemma}\label{lem:maps_pab}The set of operad maps from $\pab$ to $\mathsf{P}$ is identified with the set of triples $(m,\beta, \alpha)$ where $m \in \ob \mathsf{P}(2)$, $\beta$ is a morphism in $\mathsf{P}(2)$ from $m$ to $\sigma m$ where $\sigma = (21)$ is the non-trivial element in $\Sigma_2$, and $\alpha$ is a morphism in $\mathsf{P}(3)$ between $m \circ_1 m$ and $m \circ_2 m$. These triples are subject to the pentagon and hexagon relations spelled out below.
\end{lemma}

The hexagon relations state that the diagrams

\begin{equation}\label{hexigon_rel_1} 
\xymatrix{ & \ar[dl]_{m\circ_{1}\beta} m \circ_{1}m  \ar[dr]^-{\alpha} & \\
(213) \cdot m\circ_{1} m \ar[d]_{(213)\alpha} &&  m \circ_{2} m \ar[d]^-{\beta \circ_{2} m}\\
(213) \cdot m\circ_{2}m \ar[dr]_{(213) m\circ_{2} \beta} && (231) \cdot m \circ_1 m \ar[dl]^-{(231) \alpha} \\
& (231) \cdot m \circ_2 m &
}
\end{equation} 

and

\begin{equation}\label{hexigon_rel_2} 
\xymatrix{ 
& \ar[dl]_{m\circ_{2} \beta} m\circ_{2} m \ar[dr]^-{\alpha^{-1}} & \\
(132) \cdot m\circ_{2}m \ar[d]_{(132)\alpha^{-1}} &&  m\circ_{1} m \ar[d]^-{\beta\circ_{1} m}\\
(132) \cdot m\circ_{1} m \ar[dr]_{(132) m\circ_{1}\beta } && (312) \cdot m\circ_{2} m \ar[dl]^-{(312)\alpha^{-1}} \\
& (312) \cdot m\circ_{1} m &
}
\end{equation}
commute in $\mathsf{P}(3)$.

The pentagon relation states that the diagram
\begin{equation}
\xymatrix{ 
& \ar[dl]_{id\circ_1\alpha} ((m\circ_1m)\circ_1m) \ar[dr]^-{\alpha\circ_1 id}& \\
((m\circ_2 m)\circ_1 m) \ar[d]_{\alpha \circ_2 id} && m\circ ( m, m)\ar[ddl]^-{\alpha \circ_3 id} \\
((m\circ_1 m)\circ_2 m)\ar[dr]_{id\circ_2\alpha} &&  \\
& ((m\circ_2 m)\circ_2 m) & }
\end{equation} 
commutes in $\mathsf{P}(4)$.

\begin{proof}
This is proved in detail in \cite[Theorem 6.2.4]{FresseBook1}. One implication is easy: a map of operads from $\pab$ to $\mathsf{P}$ determines such a triple $(m, \beta, \alpha)$, where $m$ is the image of the object $(12) \in \pab(2)$, and $\beta$ and $\alpha$ (called the \emph{braiding} and the \emph{associator}, respectively) are the images of the morphisms pictured in Figure \ref{pic-gen}. The reverse implication involves a version of the coherence theorem of Mac Lane.
\end{proof}

In preparation for the definition of the Grothendieck-Teichm\"{u}ller group, we recall some standard notations. Let $Y$ be a profinite group and let $\alpha, \beta$ be elements of $Y$. Let $f$ be an element of $\widehat{\mathbb{F}}_2$, the profinite completion of the free group on two generators $x$ and $y$. Let $\sigma : \mathbb{F}_2 \to Y$ be a homomorphism defined by $\sigma(x) = \alpha$ and $\sigma(y) = \beta$. Then, we write $f(\alpha, \beta)$ for the image of $f$ under $\sigma$. (By the universal property of profinite completion, to specify a map from a group $G$ to a profinite group $Y$ is equivalent to specifying a map from the profinite completion of $G$ to $Y$.)

For $1 \leq i < j \leq n$, we follow common practice and denote by $x_{ij}$ the element of the pure braid group $\pb(n)$ given by $(\boldsymbol{\beta}_{j-1} \dots \boldsymbol{\beta}_{i+1})\boldsymbol{\beta}_{j-1}^2(\boldsymbol{\beta}_{j-1} \dots \boldsymbol{\beta}_{i+1})^{-1}$.

\begin{definition}\label{defn:GT}
The \emph{Grothendieck-Teichm\"{u}ller monoid} $\ugt$ is the monoid of endomorphisms $\sigma$ of $\widehat{\mathbb{F}}_2$ of the form
\[
\sigma(x) = x^{\lambda} \quad , \quad
\sigma(y) = f^{-1} y^{\lambda}f 
\]
for some $(\lambda, f) \in \widehat{\mathbb{Z}} \times \widehat{\mathbb{F}_2}$ satisfying the following equations:
\begin{enumerate}
\item[(I)] $f(x,y)f(y,x)=1$
\item[(II)] $f(z,x)z^mf(y,z)y^mf(x,y)x^m=1$, with $z = (xy)^{-1}$ and $m=(\lambda-1)/2$
\item[(III)] $f(x_{12},x_{23}x_{24})f(x_{13}x_{23},x_{34})=f(x_{23},x_{34})f(x_{12}x_{13},x_{24}x_{34})f(x_{12},x_{23})$. 
\end{enumerate}
The first two equations hold in $\widehat{\mathbb{F}_2}$ and the last equation holds in $\widehat{\pb}(4)$. 

The pair $(\lambda, f)$ is uniquely determined by $\sigma$. This follows from equation (III), which guarantees that $f$ belongs to the commutator subgroup of $\widehat{\mathbb{F}}_2$. The multiplication of two such pairs $(\lambda,f) \cdot (\mu,g)$ is given by
\[
\bigl(\lambda\mu, f(g x^{\mu}g^{-1}, y^\mu) \cdot g)\bigr) \; .
\]
The Grothendieck-Teichm\"uller group $\gt$ is the group of units of $\ugt$.

\end{definition}

The following proposition translates the definition $\ugt$ into operadic language.

\begin{prop}\label{prop:operadic definition of gt}
The monoid $\ugt$ is the monoid of endomorphisms of $\widehat{\pab}$ fixing the objects.
\end{prop}

\begin{proof}
(c.f. \cite[Proposition 11.1.3-11.3.4]{FresseBook1})
By Lemma \ref{lem:maps_pab}, an endomorphism of $\widehat{\pab}$ fixing the objects is uniquely specified by pair $(\beta, \alpha)$, where $\beta$ is a morphism in $\widehat{\pab}(2)$ from $(12)$ to $(21)$ and $\alpha$ is a morphism in $\widehat{\pab}(3)$ from $(12)3$ to $1(23)$, i.e. $\beta \in \widehat{\mathbb{Z}}$ and $\alpha = (n, f) \in \widehat{\mathbb{Z}} \times  \widehat{\mathbb{F}}_2$. The pair $(\beta, \alpha)$ is subject to the hexagon and pentagon relations. The hexagon relations force $n$ to be $0$ (c.f. the proof of Proposition \ref{prop: plus end to ho}). Drinfel'd shows in \cite[Section 4]{drinfeld} that equations (I) and (II) taken together are equivalent to both hexagon relations, and that equation (III) is equivalent to the pentagon relation.
\end{proof}

The goal of this section is to prove the following lemma. As in Definition~\ref{corb}, we write $\cdot$ for categorical composition.

\begin{lemma}\label{lem:maps_parb}
The set of operad maps from $\parb$ to $\mathsf{P}$ is identified with the set of pairs $(g, \tau)$ where $g = (m,\beta,\alpha)$ is an operad map from $\pab$ to $\mathsf{P}$ and $\tau$ is a morphism in $\mathsf{P}(1)$, subject to the relation that the image of $(\tau, id)$ under the map
\[
\circ_1 : \mor \mathsf{P}(1) \times \mor \mathsf{P}(2) \to \mor \mathsf{P}(2)
\]
agrees with the categorical composition $\beta \cdot \sigma\beta \cdot (id \circ (\tau,\tau))$, where $id \circ (\tau,\tau)$ is the image of $(id, \tau,\tau)$ under the operadic composition map
\[
\mor \mathsf{P}(2) \times \mor \mathsf{P}(1) \times \mor \mathsf{P}(1) \xrightarrow{} \mor \mathsf{P}(2) \; .
\]
\end{lemma}

Before we go into the proof, let us fix some notation. The elements $\tau, m, \beta$ and $\alpha$ will be the images in $\mathsf{P}$ of certain elements in $\parb$ that we now describe.  We will use boldface notation $\boldsymbol{\tau}, \boldsymbol{m}, \boldsymbol{\beta}$ and $\boldsymbol{\alpha}$ for these elements in $\parb$ and set $\boldsymbol{m} = (12)$, and $\tau$ to be the morphism in $\parb(1)$ of the form $(id, [1])$ pictured below along with the other elements

\begin{figure}[htbp]
\begin{center}
\includegraphics[scale=0.3]{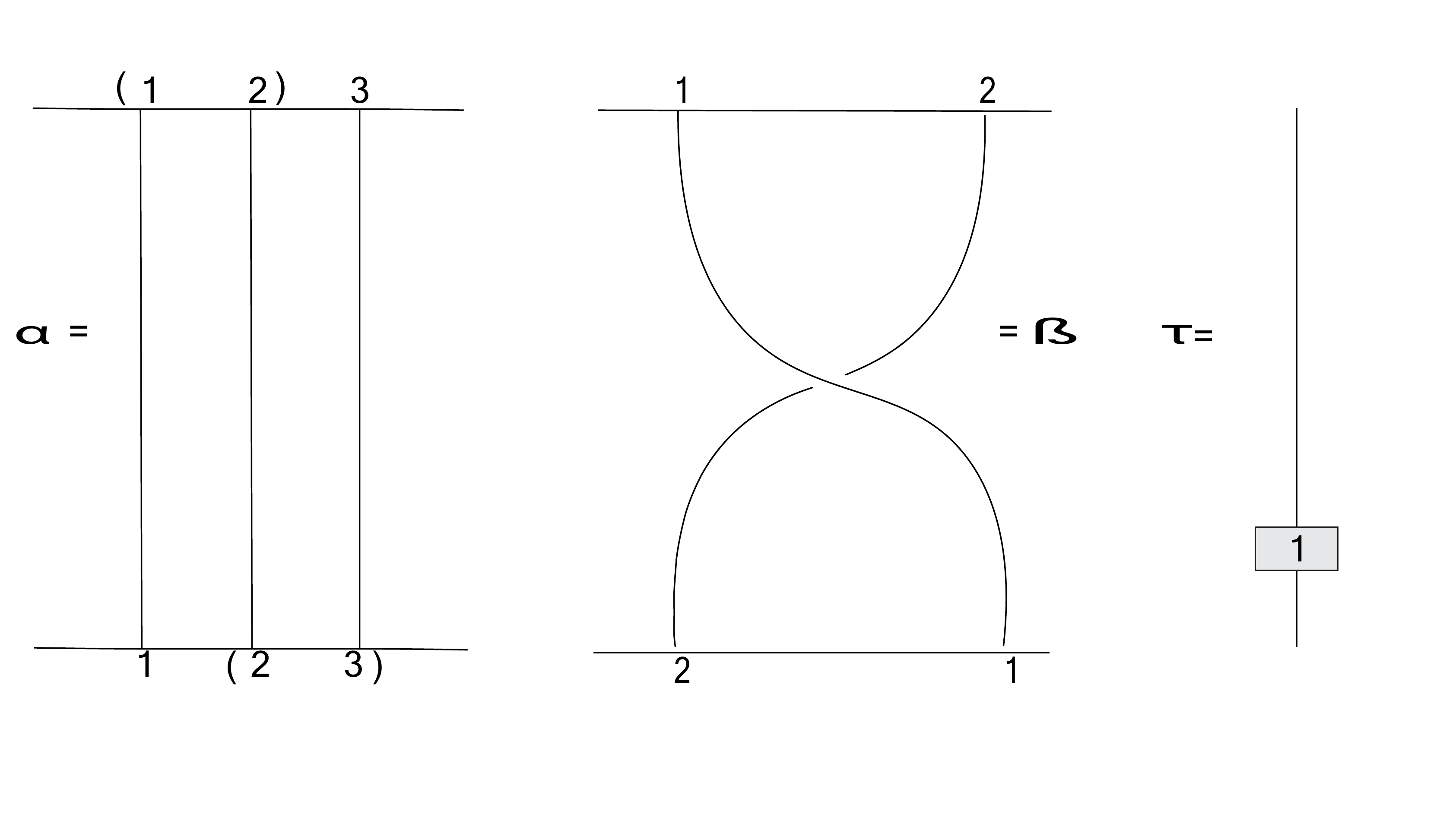}
\caption{left-to-right: $\boldsymbol{\alpha}$, $\boldsymbol{\beta}$ and $\boldsymbol{\tau}$}
\label{pic-gen}
\end{center}
\end{figure}

For a sequence of integers $n_1, \dots, n_k$, we write
\[
[n_1, \dots, n_k]
\]
for the morphism in $\parb(k)$ given by the pair $(id,[n_1, \dots, n_k])$. (This is also the image of $(id, \boldsymbol{\tau}^{n_1}, \dots, \boldsymbol{\tau}^{n_k})$ under the operadic composition map 
\[
\mor{\parb(k)} \times \mor{\parb(1)} \times \dots \times \mor{\parb(1)} \to \mor{\parb(k)}
\]
where $\boldsymbol{\tau}^m$ denotes the $m$-fold categorical composition of $\tau$.) The reader should note that there is one such morphism for each object of $\parb(k)$, however, we do not include the object in the notation in order to keep it as light as possible.

\begin{proof}
Starting with the map $g : \pab \to \mathsf{P}$, we want to lift it to a map $f : \parb \to \mathsf{P}$ such that pre-composition of $f$ with the canonical map $\pab \to \parb$ is $g$. We first define $f(k) : \parb(k) \to \mathsf{P}(k)$ for each $k$ using $g$ and $\tau$, and then use the relation in the statement of the lemma to show that the definition is indeed a map of operads. 

Firstly, it is enough to describe $f(k)$ on morphism sets; the definition extends to object sets via the source-target maps. Recall that a morphism in $\parb(k)$ is a pair $(\gamma, [n_1, \dots, n_k])$, where $\gamma$ is a morphism in $\pab(k)$ and each $n_i$ is an integer, and that the \emph{categorical} composition in $\parb(k)$ separates the braid concatenation and the addition of twists on each strand. Therefore, any morphism $(\gamma, [n_1, \dots, n_k])$ in $\parb(k)$ has a unique expression as a categorical composition of the form $(\gamma, [0,\dots, 0]) \cdot (id, [n_1, \dots, n_k])$, which we abbreviate as $\gamma \cdot [n_1, \dots, n_k]$.

Now we come to the definition of $f$. We declare $f(\boldsymbol{\tau}) = \tau$, 
\[
f([n_1, \dots, n_k]) = g(id) \circ (\tau^{n_1}, \dots, \tau^{n_k})
\]
and
\[
f(\gamma \cdot [n_1, \dots, n_k]) = g(\gamma) \cdot f([n_1, \dots, n_k]) \; .
\]
One easily checks that $f(k)$ is a map of groupoids for each $k$.

We now check that $f$ is a map of operads. Let $\gamma \cdot [n_1, \dots, n_k]$ and $\gamma^\prime \cdot [m_1, \dots, m_\ell]$ be elements in $\mor{\parb(k)}$ and $\mor{\parb(\ell)}$ respectively. We need to show that
\begin{equation}\label{eq:isanoperad}
f(\gamma \cdot [n_1, \dots, n_k] \circ_i \gamma^\prime \cdot [m_1, \dots, m_\ell]) = f(\gamma \cdot [n_1, \dots, n_k]) \circ_i f(\gamma^\prime \cdot [m_1, \dots, m_\ell])
\end{equation}

Writing $\gamma \cdot [n_1, \dots, n_k] \circ_i \gamma^\prime \cdot [m_1, \dots, m_\ell]$ as a categorical composition of $\gamma \circ_i \gamma^\prime$ and $[n_1, \dots, n_k] \circ_i [m_1, \dots, m_\ell]$, we can express the left-hand side of (\ref{eq:isanoperad}) as
\[
g(\gamma \circ_i \gamma^\prime) \cdot f([n_1, \dots, n_k] \circ_i [m_1, \dots, m_\ell]) \; .
\]
On the other hand, using that $\mathsf{P}$ is an operad, the right-hand side of (\ref{eq:isanoperad}) is equal to
\[
\bigl ( g(\gamma) \circ_i g(\gamma^\prime) \bigr) \cdot \bigl ( f([n_1, \dots, n_k]) \circ_i f([m_1, \dots, m_\ell]) \bigr) \; .
\]
Now, since $g$ is a map of operads we know that $g(\gamma \circ_i \gamma^\prime) = g(\gamma) \circ_i g(\gamma^\prime)$. Therefore, the equations (\ref{eq:isanoperad}) hold if and only if the equations
\[
f([n_1, \dots, n_k] \circ_i [m_1, \dots, m_\ell]) = f([n_1, \dots, n_k]) \circ_i f([m_1, \dots, m_\ell])
\]
hold. At this point, we remind the reader that $[n_1,\ldots,n_k]$ is notation for an automorphism of some unspecified object in $\parb(k)$; thus, the equation above is really a collection of equations, one for each choice of objects that makes the source and target of both sides coincide.

There are further reductions to be made. As a first step, by expressing $[n_1, \dots, n_k]$ as $[n_1, \dots, n_k] \cdot id_k$ and $[m_1, \dots, m_\ell]$ as $id_\ell \cdot [m_1, \dots, m_\ell]$, we split the problem into two:
\begin{equation}\label{eq: easyproblem}
f(id_k \circ_i [m_1, \dots, m_\ell]) = f(id_k) \circ_i f([m_1, \dots, m_\ell])
\end{equation}
and
\begin{equation}\label{eq: hardproblem}
f([n_1, \dots, n_k] \circ_i id_\ell) = f([n_1, \dots, n_k]) \circ_i f(id_\ell) \; ,
\end{equation}
where we use the notation $id_k$ or $id_\ell$ to denote the identity of an unspecified object in arity $k$ or $\ell$.
By definition of $f$, the equation (\ref{eq: easyproblem}) is always satisfied.

In order to check equation (\ref{eq: hardproblem}), we do a further reduction. Using that
\[
[n_1, \dots, n_k] = [n_1, 0, \dots, 0] \cdot [0, n_2, 0, \dots, 0] \cdot \, \dots \, \cdot [0, \dots, 0, n_k] 
\]
we may assume that, for a given $i$, $n_j$ is zero for all $j \neq i$ and $n_i = 1$. Thus, equation (\ref{eq: hardproblem}) holds if and only if the equation
\[
f([0,\dots, 1, \dots, 0] \circ_i id_\ell) = f([0,\dots, 1, \dots, 0]) \circ_i f(id_\ell) \; ,
\]
holds, where $1$ is in the $i^{th}$ position.

To proceed, we look at a generalization of equation (\ref{eq: easyproblem}): for any morphism $S \in \mor{\parb(\ell)}$, we have that 
$
f(id_k \circ_i S) = f(id_k) \circ_i f(S)
$. Using this, and the fact that $[0,\dots, 1, \dots, 0] = id_k \circ_i [1]$, we deduce that equation (\ref{eq: hardproblem}) holds if and only the equation
\begin{equation}\label{eq:tl}
f(\boldsymbol{\tau} \circ id_\ell) = f(\boldsymbol{\tau}) \circ f(id_\ell)
\end{equation}
holds (recall that $[1] = \boldsymbol{\tau}$).

By assumption, equation (\ref{eq:tl}) holds when $\ell = 2$. Indeed, the relation $\boldsymbol{\tau}\circ id_2 = \boldsymbol{\beta} \cdot \sigma \boldsymbol{\beta} \cdot (id \circ (\boldsymbol{\tau},\boldsymbol{\tau}))$ holds in $\parb$ and, applying $f$, we obtain the identity
\[
f(\boldsymbol{\tau}\circ id_2) = f\bigl(\boldsymbol{\beta} \cdot \sigma \boldsymbol{\beta} \cdot (id \circ (\boldsymbol{\tau},\boldsymbol{\tau})) \bigr) \; .
\]
By definition of $f$, the right-hand side equals $\beta \cdot \sigma \beta \cdot (id \circ (\tau,\tau))$ and so, by hypothesis, it equals $\tau\circ id_2$.

Now assume that equation (\ref{eq:tl}) has been established (for a chosen $\ell$). By associativity of operadic composition, we have equations
\[(\boldsymbol{\tau}\circ id_\ell)\circ_j id_2=\boldsymbol{\tau}\circ  id_{\ell+1}\;\; \textrm{and}\;\;(\tau\circ id_\ell)\circ_j id_2=\tau\circ  id_{\ell+1} \;.\]

Since $\ob(\parb)$ is generated by arity $2$ operations, any object in $\parb(\ell+1)$ is in the image of at least one of the composition maps:
\[\circ_j:\parb(\ell)\times\parb(2)\to\parb(\ell+1)\]
In particular, the equation
\[f(\boldsymbol{\tau}\circ id_{\ell+1})=\tau\circ id_{\ell+1}\]
holds if and only if, for each $j$, the equation
\begin{equation}\label{eq: final equation}
f((\boldsymbol{\tau}\circ id_\ell)\circ_j id_2)=(\tau\circ id_\ell)\circ_j id_2
\end{equation}
holds. But as we have already observed, the equation $f(\boldsymbol{\tau}\circ id_2)=\tau\circ id_2$ holds. If we reverse all the steps that let us reduce equation (\ref{eq:isanoperad}) to equation (\ref{eq:tl}), we see that (\ref{eq:isanoperad}) holds when $\ell=2$. In particular, equation (\ref{eq: final equation}) holds for each $j$.
\end{proof}

\subsection{Endomorphisms of $\parb$ fixing the objects}

We write $\End_0(\parb)$ and $\End_0(\pab)$ for the set of endomorphisms of ${\parb}$ and $\pab$ which are the identity on objects.

\begin{definition} 
The operad of twists $\mathsf{T}$ is the operad in groupoids which in degree $n$ is the group $\mathbb{Z}^n$ seen as a groupoid with a unique object. The operadic composition 
\[\circ_i:\mathsf{T}(n)\times\mathsf{T}(m)\to\mathsf{T}(n+m-1)\]
is given by the formula
\[(a_1,\ldots,a_n)\circ_i(b_1,\ldots,b_m):=(a_1,\ldots,a_{i-1},a_i+b_1,\ldots,a_i+b_m,a_{i+1},\ldots,a_n)\]
\end{definition} 

Note that there is a trivial morphism from any operad in groupoids to $\mathsf{T}$ which sends any morphism to $(0,\ldots,0)$.

\begin{lemma}\label{lemm: maps pab to T}
The trivial morphism is the only morphism from $\pab$ to $\mathsf{T}$.
\end{lemma}

\begin{proof}
By Lemma \ref{lem:maps_pab}, such a morphism is determined by where it sends $\boldsymbol{\beta}$ and $\boldsymbol{\alpha}$. Let us call the image of these two morphisms $(x,y)\in\mathbb{Z}^2$ and $(a,b,c)\in\mathbb{Z}^3$ respectively. The pentagon relation gives us
\[(2a,2b+a,c+2b,2c)=(2a,a+b,b+c,2c)\]
which implies that $b=0$. The two hexagon relations force $a,c,x,y$ to all be $0$.
\end{proof}

\begin{prop}\label{prop: endomorphisms fixing the objects}
The following holds:
\begin{enumerate}
\item Any endomorphism of $\parb$ fixing the objects has to preserve $\pab\subset\parb$.
\item The induced map
\[\End_0(\parb) \to \End_0(\pab)\]
is an isomorphism.
\end{enumerate}
\end{prop}

\begin{proof}
Using Lemma \ref{lem:maps_parb} and the notation therein, an endomorphism of $\parb$ fixing the objects is uniquely specified by a pair $(g, \tau)$, where $g = (m, \beta, \alpha)$ represents a map from $\pab$ to $\parb$ fixing the objects and $\tau$ is a morphism in $\parb(1)$. We claim that the map $g$ has to send $\pab$ to $\pab$. Indeed, there is a nontrivial map $\parb\to\mathsf{T}$ that sends a ribbon braid to the list of its twists. We can form the composite
\[\pab\xrightarrow{g}\parb\to\mathsf{T}\]
and according to \ref{lemm: maps pab to T} such a map has to be the trivial map. Therefore $g$ factors through the inclusion $\pab\to\parb$ (viewing $\pab$ as the suboperad whose morphisms have no twists). This proves (1).

By taking the underlying braiding together with the number of twists on each strand, the set $\mor_{\parb(2)}(12, 12)$ is identified with $\mor_{\pab(2)}(12, 12) \times \mathbb{Z} \times \mathbb{Z}$, i.e. $2\mathbb{Z} \times \mathbb{Z} \times \mathbb{Z}$. Recall that the set of morphisms in $\pab(2)$ from $(12)$ to $(21)$ agrees with $\mathbb{Z}$, the braid group on two strands. The morphism $\beta \cdot \sigma \beta$ is thus given by a triple $(2\beta_1, \beta_2, \beta_3)$ in $2\mathbb{Z} \times \mathbb{Z} \times \mathbb{Z}$ and $\tau$ is given by a single integer. By part (1), $\beta_2$ and $\beta_3$ have to be zero. The relation
\begin{equation}\label{eq:parbrelation}
\tau \circ id = \beta \cdot \sigma \beta \cdot (id \circ (\tau, \tau))
\end{equation}
which holds in the set of morphisms on $\parb(2)$ from $(12)$ to $(12)$, may then be expressed as a relation
\[
(2\tau, \tau, \tau) = (2\beta_1, 0,0) \cdot (0,\tau,\tau) \; .
\]
Therefore, $\beta_1 = \tau$. Hence we can construct a map $\End_0(\pab)\to\End_0(\parb)$ sending $g$ to $(g,\beta_1)$ which is an inverse to the restriction map.
\end{proof}

\subsection{Endomorphisms of $\parb$ up to homotopy}

The category of operads in groupoids is cotensored over groupoids. It follows that we can define a homotopy between two maps of operads in groupoids. We denote by $E$ the groupoid completion of $[1] = \{0 < 1\}$ and by $s$ and $t$ the two maps $[0]\to E$. For $f,g:\mathsf{P}\to\mathsf{Q}$ two maps of operads, a homotopy between $f$ and $g$ is a map $H:\mathsf{P}\to\mathsf{Q}^{E}$ such that when we postcompose with the two maps $\mathsf{Q}^{E}\to\mathsf{Q}$ induced by $s$ and $t$, we recover $f$ and $g$. The relation ``being homotopic'' is an equivalence relation between maps from $\mathsf{P}$ to $\mathsf{Q}$ and this equivalence relation is compatible with composition of morphisms. It follows that for $\mathsf{P}$ an operad in groupoids, the set of endomorphisms of $\mathsf{P}$ up to homotopy gets a monoid structure. We denote that monoid by $\hoEnd(\mathsf{P})$.

\begin{prop}\label{prop:0toho}
The composition
\[\End_0(\parb)\to\End(\parb)\to\hoEnd(\parb)\]
is an isomorphism.
\end{prop}

\begin{proof}
The surjectivity of this map can be translated by saying that any endomorphism of $\parb$ is homotopic to one that fixes the objects. This can be proven exactly as in \cite[Theorem 7.8.]{Horel1}.

Now, we prove injectivity. Let us denote by $\End_0(\prb(3))$ the monoid of endomorphisms of $\prb(3)$ that preserve the subgroup $\pb(3)$ and by $\hoEnd_0(\prb(3))$ the monoid of endomorphisms of $\prb(3)$ that preserve the subgroup $\pb(3)$ modulo homotopies.  We construct a commutative diagram
\[\xymatrix{
\End_0(\parb)\ar[r]\ar[d]&\hoEnd_0(\parb)\ar[r]^f&\hoEnd_0(\prb(3))\ar[d]\\
\End_0(\pab)\ar[r]&\hoEnd_0(\pab)\ar[r]_g&\hoEnd(\pb(3))
}
\]
The map labeled $f$ is induced by the restriction map
\[\End_0(\parb)\to\End_0(\parb(3))\to\End_0(\prb(3))\]
where the first map is the restriction to arity $3$ and the second map is the restriction to an object in $\parb(3)$. The map labeled $g$ is defined analogously. 

The right-hand vertical map is obtained by restriction to the subgroup $\pb(3)$. This is well-defined, as we now explain. Let $u$ and $v$ be two endomorphisms of $\prb(3)\cong\pb(3)\times\mathbb{Z}^3$ fixing the subgroup $\pb(3)\subset\prb(3)$ and such that there exists an element $h$ in $\prb(3)$ such that $u(x)=h^{-1}v(x) h$. Since the subgroup $\mathbb{Z}^3$ in $\prb(3)$ is contained in the center, we may assume without loss of generality that $h$ lies in $\pb(3)$ and we deduce that the restrictions of $u$ and $v$ to $\pb(3)$ are conjugate. The commutativity of the diagram is immediate.

By \cite[Proposition 7.7]{Horel1}, the lower horizontal composite is injective. By Proposition \ref{prop: endomorphisms fixing the objects}, the left vertical arrow is an isomorphism. It follows that the map $\End_0(\parb)\to\hoEnd_0(\parb)$ is injective as desired.
\end{proof}

\section{The main theorem}

Propositions \ref{prop: endomorphisms fixing the objects} and \ref{prop:0toho} have profinite variants which we state below. The proofs are similar.
\begin{prop} 
The map 
\[\End_0(\widehat{\parb}) \to \End_0(\widehat{\pab})\]
and the composite
\[\End_0(\widehat{\parb})\to\End(\widehat{\parb})\to\hoEnd(\widehat{\parb})\]
is an isomorphism of monoids. 
\end{prop}

By Proposition \ref{prop:operadic definition of gt}, the monoid $\End_0(\widehat{\pab})$ is isomorphic to $\ugt$, the Grothendieck-Teichm\"uller monoid, it follows that
\begin{equation}\label{eq:gtisho}
\ugt \cong \hoEnd(\widehat{\parb}) \; .
\end{equation}

\begin{prop}\label{prop:hoend}
The monoid $\hoEnd(\widehat{\parb})$ is isomorphic to the monoid of path components of
$\Rmap(N \widehat{\parb}, N\widehat{\parb})$.
\end{prop}
\begin{proof}
We claim that the statement holds for any operad $\mathsf{P}$ in groupoids which, like $\parb$, is cofibrant and such that $\mathsf{P}(n)$ has finitely many objects for each $n$. For such an operad, the set $\hoEnd(\widehat{\mathsf{P}})$ is identified with the set of path components of
\[
\map(\mathsf{P}, |\widehat{\mathsf{P}}|)
\]
where $\map$ refers to the mapping space in the category of operads in groupoids. Since the dendroidal nerve functor is homotopically fully faithful, the map
\[
\Rmap(\mathsf{P}, |\widehat{\mathsf{P}}|) \to \Rmap(N \mathsf{P}, N|\widehat{\mathsf{P}}|)
\]
is a weak equivalence of spaces. The dendroidal space $N|\widehat{\mathsf{P}}|$ is an $
\infty$-operad by Remark \ref{rem:Nisopgp}. The right-hand mapping space is identified with
$$\Rmap(\widehat{N \mathsf{P}}, N\widehat{\mathsf{P}})$$
since $|-|$ and $N$ commute and the profinite completion functor agrees with its left (and right) derived functor since it preserves all weak equivalences. Moreover, $\widehat{N \mathsf{P}} \cong N \widehat{\mathsf{P}}$ since completion of groupoids with finitely many objects commutes with products (Remark \ref{rem:Nisopgp}).
\end{proof}

Given an operad in (profinite) groupoids $G$, we let $BG$ denote the operad in (profinite) spaces obtained via the classifying space construction. There is a natural map
\begin{equation}\label{eq:goodmap}
(B N \parb)^{\wedge} \to B N \widehat{\parb}
\end{equation}
where the left-hand side is an alternative notation for the profinite completion of $B N \parb$. This map is given as the adjoint of the composite
\[
BN \parb \to BN|\widehat{\parb}| \xrightarrow{\cong} | B N \widehat{\parb}| \;.
\]
where the first map is the unit of the adjunction between operads in groupoids and operads in profinite groupoids.

\begin{lemma}\label{lemm: goodness of pure ribbon braid groups}
The map (\ref{eq:goodmap}) is a weak equivalence.
\end{lemma}
\begin{proof}
The pure ribbon braid groups $\prb(n)$ are good since they split as a product of good groups $\pb(n) \times \mathbb{Z}^n$ and so we can apply \cite[Corollary 5.11, Corollary 5.12]{Horel1}. Thus, by Proposition \ref{prop:Nisop}, both sides are $\infty$-operads in profinite spaces. It suffices to prove that the map is a weak equivalence on corollas, i.e. that $$(B\prb(n))^{\wedge} \to B \widehat{\prb(n)}$$ is a weak equivalence for every $n$. This follows again from the fact that the pure ribbon braid groups are good.
\end{proof}

Putting it all together, we obtain
\begin{thm}\label{mainthm}

There is an isomorphism
\[
\ugt \cong \pi_0 \REnd(\widehat{\mathcal{M}})
\]
where $\mathcal{M}$ denotes the version of the genus zero surface operad without $0$-arity operations.
\end{thm}
\begin{proof}
By Lemma \ref{lem:parbisfE2},  $ \REnd(\widehat{\mathcal{M}})$ is weakly equivalent to $\REnd((BN \parb)^{\wedge})$. The latter is weakly equivalent to $\REnd(BN\widehat{\parb})$ by Lemma \ref{lemm: goodness of pure ribbon braid groups}. The classifying space functor $B$ induces a homotopically fully faithful functor from $\infty$-operads in profinite groupoids to $\infty$-operads in profinite spaces. In particular, the map $$\Rmap(N \widehat{\parb}, N\widehat{\parb}) \to \Rmap(B N \widehat{\parb}, B N\widehat{\parb})$$
is a weak equivalence of spaces. By Proposition \ref{prop:hoend} and the isomorphism (\ref{eq:gtisho}), the monoid of path components of the source is isomorphic to $\ugt$.
\end{proof}


\section{Formality of the genus zero surface operad}\label{section:formality}

It has been proved independently by Severa and Giansiracusa-Salvatore (see \cite{giansiracusaformality,severaformality}) that the framed little disks operad is rationally formal. That is to say, there is a zigzag of quasi-isomorphisms of dg-operads between $C_*(\mathsf{FD},\mathbb{Q})$ and its homology, seen as a dg-operad with zero differential. In this section, we exploit the action of the Grothendieck-Teichm\"uller group on the profinite completion of the genus zero surface operad in order to give an alternative proof of the formality of $\mathcal{M}$ and, equivalently, $\mathsf{FD}$. The idea is to use the fact that there is a model for $C_*(\mathcal{M},\mathbb{Q}_p)$ that is computed using the profinite completion of $\mathcal{M}$ and thus inherits a $\gt$-action. This large supply of automorphisms on the chains on $\mathcal{M}$ allows us to apply a formality criterion introduced by Guillen Navarro Pascual and Roig in \cite{guillenmoduli}

In preparation for our proof, we introduce a notation. For $X=\mathrm{lim}_iX_i$ a pro-simplicial set and $R$ a commutative ring, we denote by $C^\bullet(X,R)$ the cosimplicial $R$-module given by the formula 
\[C^\bullet(X,R):=\mathrm{colim}_i C^\bullet(X_i,R) \; . \]
There is a K\"unneth isomorphism at the level of cosimplicial objects in the sense that there is a natural isomorphism
\[C^\bullet(X\times Y,R)\cong C^\bullet(X,R)\otimes_RC^\bullet(Y,R) \; .\]
In particular, if $\mathsf{Q}$ is an operad in pro-simplicial sets, $C^\bullet(\mathsf{Q},R)$ has the structure of a cosimplicial cooperad in $R$-modules.

\begin{thm}\label{thm: formality}
The operad $\mathcal{M}$ is formal, that is, there exists a zig-zag of quasi-isomorphisms of dg-operads in $\mathbb{Q}$-vector spaces.
\[C_*(\mathcal{M},\mathbb{Q})\leftarrow X\rightarrow H_*(\mathcal{M},\mathbb{Q})\]
\end{thm}

\begin{proof}
We follow the strategy of \cite{petersenformality}. First, by \cite[Theorem 6.2.1.]{guillenmoduli}, it suffices to prove that $C_*(\mathcal{M},\mathbb{Q}_p)$ is formal as an operad in dg-operad in $\mathbb{Q}_p$-vector spaces. For any positive integer $n$, there is a quasi-isomorphism of cosimplicial cooperads in $\mathbb{Z}/p^n$-modules.
\[ C^\bullet(\mathcal{M},\mathbb{Z}/p^n)\simeq C^\bullet(B\parb,\mathbb{Z}/p^n)\simeq C^\bullet(B\widehat{\parb},\mathbb{Z}/p^n)\]

Taking the limit over $n$ (which in this case is a homotopy limit since the transition maps are surjections), we get a quasi-isomorphism of cosimplicial cooperads
\[\lim_n C^\bullet(\mathcal{M},\mathbb{Z}/p^n )\simeq \mathrm{lim}_n C^\bullet(B\widehat{\parb},\mathbb{Z}/p^n)\]
We also claim that the map
\[C^\bullet(\mathcal{M},\mathbb{Z}_p)\to \lim_n C^\bullet(\mathcal{M},\mathbb{Z}/p^n)\]
is a quasi-isomorphism as can be seen from Milnor's short exact sequence and the fact that the  cohomology of the spaces $\mathcal{M}(n)$ is finitely generated which implies that the Mittag-Leffler condition holds. Tensoring with $\mathbb{Q}_p$ we get a quasi-isomorphism of cosimplicial cooperads over $\mathbb{Q}_p$,
\[C^\bullet(\mathcal{M},\mathbb{Q}_p)\simeq (\mathrm{lim}_n C^\bullet(B\widehat{\parb},\mathbb{Z}/p^n))\otimes_{\mathbb{Z}_p}\mathbb{Q}_p \; .\]
After dualizing, the universal coefficient theorem, gives us a quasi-isomorphism of simplicial operads
\[C_\bullet(\mathcal{M},\mathbb{Q}_p)\simeq C^\bullet(\mathcal{M},\mathbb{Q}_p)^{\vee}\simeq ((\mathrm{lim}_n C^\bullet(B\widehat{\parb},\mathbb{Z}/p^n))\otimes_{\mathbb{Z}_p}\mathbb{Q}_p)^{\vee} \; .\]
We denote by $\mathsf{P}$ the underlying dg-operad of the simplicial operad on the right-hand side. Our goal is to show that the dg-operad $\mathsf{P}$ is formal. By our main theorem, $\mathsf{P}$ has an action of the group $\gt$. We claim that the induced map
\[\gt\to\mathrm{Aut}(H_*(\mathsf{P}))\]
factors as the composite of the cyclotomic character
\[\chi:\gt \to\widehat{\mathbb{Z}}^{\times}\to\mathbb{Z}_p^{\times}\]
with the map
\[\mathbb{Z}_p^{\times}\to\mathrm{Aut}(H_*(\mathsf{P}))\]
sending $u\in\mathbb{Z}_p^{\times}$ to the automorphism $\phi_u$ of $H_*(\mathsf{P})$ that acts as multiplication by $u^n$ in homological degree $n$. Firstly, it is well known that $H_*(\mathsf{P})$ is the operad $\mathsf{BV}$ of Batalin-Vilkovisky algebras. It is generated by a commutative algebra product in arity $2$ and degree $0$ and an operator $\Delta$ in arity $1$ and homological degree $1$. Therefore, it suffices to prove the claim on these two homology groups. It is straightforward that $\gt$ acts trivially on $H_0(\mathsf{P}(2))$. The action of $\gt$ on $B\widehat{\prb(1)}\cong B\widehat{\mathbb{Z}}$ is given precisely by inducing the obvious action of $\widehat{\mathbb{Z}}^\times$ on $\widehat{\mathbb{Z}}$ along the projection $\gt\to\widehat{\mathbb{Z}}^\times$ (see the proof of Proposition \ref{prop: endomorphisms fixing the objects}(2)). It follows that the action of $\gt$ on $H_1(\mathsf{P})$ is the desired action. This proves the claim.

Now, we follow the strategy explained by Petersen in \cite[Proposition p. 819] {petersenformality}. We pick an infinite order unit $u$ in $\mathbb{Z}_p$. Since the cyclotomic character map $\chi:\gt\to\mathbb{Z}_p^{\times}$ is surjective, we can find an automorphism of $\mathsf{P}$ that induces the grading automorphism $\phi_u$ on the homology.
\end{proof}


\section{An action of GT on the operad of compactified moduli spaces}

For $n \geq 3$, the moduli space $\mathcal{M}_{0,n}$ of compact complex algebraic curves of genus zero with $n$ punctures is identified with the space of configurations of $n$ distinct points on the complex projective line $\mathbb{CP}^1$ modulo the action of $PGL_2(\mathbb{C})$. The Deligne-Knudsen-Mumford compactification of $\mathcal{M}_{0,n}$, denoted $\bmod_{0,n}$, is the space of isomorphism classes of stable $n$-punctured complex curves of genus zero. By convention, $\bmod_{0,2} = *$.

The collection of moduli spaces $\bmod_{0,\bullet+1}:=\{\bmod_{0,n+1}\}_{n\geq 1}$ forms an operad in spaces with no arity zero term \cite{GetzlerModuli}. For a curve in $\bmod_{0,n+1}$, we consider the first $n$ points as inputs and the last point as the output. The symmetric group $\Sigma_n$ acts on $\bmod_{0,n+1}$ by permuting the labels of the inputs and leaving the output untouched. Operad composition 
\[
\circ_k : \bmod_{0, n+1} \times \bmod_{0,m+1} \to \bmod_{0,n+m}
\]
is given by attaching the output of $\bmod_{0,m+1}$ to the $k^{th}$ input in $\bmod_{0,n+1}$ and creating a new genus zero stable curve with one additional double point. 

A theorem of Drummond-Cole relates the framed disks operad to the operad $\bmod_{0,\bullet+1}$ via a homotopy pushout diagram
\begin{equation}\label{drummondcole}
\xymatrix{
S^1\ar[d]\ar[r]&\mathsf{FD}\ar[d]\\
\ast\ar[r]&\bmod_{0,\bullet +1}
}
\end{equation}
in the category of operads in spaces, where $S^1$ and $\ast$ denote the topological groups $S^1$ and $\ast$ seen as operads concentrated in arity $1$ and the map $S^1\to\mathsf{FD}$ is the inclusion of arity $1$ operations. Given that the operad $\mathsf{FD}$ is homotopy equivalent to $\mathcal{M}$ we can replace $\mathsf{FD}$ in the homotopy pushout square. 

\begin{prop}
The homotopy groups of $\bmod_{0,n}$ are good groups.
\end{prop}

\begin{proof}
The spaces $\bmod_{0,n}$ are simply connected compact complex manifolds. As such their homotopy groups are finitely generated abelian groups. More generally, any finitely generated abelian group is good (see for instance \cite[p.5]{sullivangenetics}). Indeed, such a group is a finite product of copies of $\mathbb{Z}$ and $\mathbb{Z}/n$ for various $n$'s.   A finite group is automatically good and $\mathbb{Z}$ is good. The claim then follows from \cite[Proposition 5.10.]{Horel1}.
\end{proof}

\begin{cor}\label{coro:Misop}
The dendroidal profinite space $(N\bmod_{0,\bullet+1})^{\wedge}$ is an $\infty$-operad.
\end{cor}

\begin{proof}
This follows from the previous proposition and Proposition \ref{prop:Nisop}.
\end{proof}

The main goal of this section is to prove the following. 
\begin{thm}\label{thm: action on bmod}
There exists an action of $\ugt$ on the $\infty$-operad $(N\bmod_{0,\bullet+1})^{\wedge}$ that makes the map 
\[(N\mathcal{M})^{\wedge}\to(N\bmod_{0,\bullet+1})^{\wedge}\]
into a $\ugt$-equivariant map.
\end{thm}

\begin{proof}
The functor $N$ is a right Quillen equivalence (Theorem \ref{thm:CM}) and hence it preserves homotopy pushout squares. Therefore, the square
\[
\xymatrix{
N S^1\ar[d]\ar[r]& N\mathcal{M}\ar[d]\\
\ast\ar[r]& N\bmod_{0,\bullet+1}
}
\]
is a pushout square of $\infty$-operads in spaces. Applying the profinite completion functor levelwise, we obtain a square of dendroidal objects in profinite spaces. In fact, each term of this new square is an $\infty$-operad in profinite spaces, by Proposition \ref{prop:Nisop}. We claim that this square of $\infty$-operads is a pushout square in the $\infty$-category (relative category) of $\infty$-operads in profinite spaces. This is a consequence of the following formal observation. Given a pushout square in the $\infty$-category of $\infty$-operads in spaces, consider the resulting square obtained by applying profinite completion levelwise. Then this square is a pushout in the $\infty$-category of $\infty$-operads in profinite spaces if each of its terms is an $\infty$-operad. To see this, one can use the fact that the hypothetical pushout has the correct universal property in the $\infty$-category of $\infty$-operads in profinite spaces.

We can now prove the statement of the theorem. The top horizontal map of the square is the inclusion of arity one operations. It follows that the action of $\ugt$ on $(N\mathcal{M})^{\wedge}$ restricts to an action of $(NS^1)^{\wedge}$ in such a way that this map becomes a $\ugt$-equivariant map. On the other hand, the map $(NS^1)^{\wedge}\to \ast$ is obviously $\ugt$-equivariant for the trivial action on $\ast$. It follows that the $\infty$-operad $(N\bmod_{0,\bullet+1})^{\wedge}$ inherits a $\ugt$ action that makes the square $\ugt$-equivariant.
\end{proof}

We now want to prove that the action constructed in the previous theorem is non-trivial. In order to do so, we will prove that this action is non-trivial after application of $H_*(-,\mathbb{Q}_p)$. First we need to explain what we mean by $H_*(X,\mathbb{Q}_p)$ when $X$ is a profinite space. 

\begin{cons}\label{cons: functor D}
We have explained in the previous section how to construct a cosimplicial $\mathbb{Z}/p^n$-module $C^\bullet(X,\mathbb{Z}/p^n)$. We can then define $C^*(X,\mathbb{Z}/p^n)$ as the associated cochain complex. Define the chain complex
\[D_*(X,\mathbb{Q}_p):=((\lim_n C^*(X,\mathbb{Z}/p^n))\otimes_{\mathbb{Z}_p}\mathbb{Q}_p)^{\vee} \; .\]
As in the proof of Theorem \ref{thm: formality}, one can show that $D_*(\widehat{Y},\mathbb{Q}_p)$ is naturally quasi-isomorphic to $C_*(Y,\mathbb{Q}_p)$ when $Y$ is a space with finitely generated homology. We denote the homology of $D_*(X,\mathbb{Q}_p)$ by $H_*(X,\mathbb{Q}_p)$. Since $C^*(-,\mathbb{Z}/p^n)$ sends homotopy colimits to homotopy limits, we deduce that $D_*(-,\mathbb{Q}_p)$ preserves homotopy colimits.
\end{cons}

To make explicit the action of $\gt$ on $H_*(\widehat{\mathcal{M}},\mathbb{Q}_p)$ we use the $p$-adic cyclotomic character $\chi_p:\gt\to \widehat{\mathbb{Z}}^{\times}\to\mathbb{Z}_p^\times$. 

\begin{prop}\label{prop: action on H}
Let $g$ be an element of $\gt$. Then the action of $g$ on the vector space $H_i(\widehat{\mathcal{M}(n)},\mathbb{Q}_p)$ is given by multiplication by $\chi_p(g)^{i}$.
\end{prop}

\begin{proof}
This vector space is isomorphic as a $\gt$-representation to the vector space $H_i(\mathsf{P}(n))$ appearing in the proof of Theorem \ref{thm: formality}. The desired statement can be found in that proof.
\end{proof}

\begin{prop}
The action of $\gt$ on $\bmod_{0,\bullet +1}$ is non-trivial.
\end{prop}

\begin{proof}
A standard argument with simplicial model categories applied to the homotopy pushout square (\ref{drummondcole}) tells us that the operad $\bmod_{0,\bullet+1}$ is the homotopy colimit of the simplicial diagram
\[[n]\mapsto \mathcal{M}\sqcup (S^1)^{\sqcup n}\sqcup \ast\]
where $\sqcup$ denotes the coproduct in the category of operads. Applying the dendroidal nerve functor followed by the profinite completion we get a simplicial diagram
\[[n]\mapsto N(\mathcal{M}\sqcup (S^1)^{\sqcup n}\sqcup \ast)^{\wedge}\]
in the relative category $\mathbf{d}\widehat{\mathbf{S}}$ whose homotopy colimit computes $(N\bmod_{0,\bullet+1})^\wedge$. Indeed, since the category $\Delta$ is sifted, this homotopy colimit coincides with the homotopy colimit computed in $\mathbf{Op}_{\infty}(\widehat{\mathbf{S}})$. 

Evaluating at the corolla $C_n$, we get a simplicial profinite space
\[[n]\mapsto N(\mathcal{M}\sqcup (S^1)^{\sqcup n}\sqcup \ast)^{\wedge}_{C_n}\]
whose homotopy colimit is $(\bmod_{0,n+1})^\wedge$. We can hit this diagram with the functor $D_*(-,\mathbb{Q}_p)$ constructed in \ref{cons: functor D} and we get a simplicial chain complex
\[[n]\mapsto D_*(N(\mathcal{M}\sqcup (S^1)^{\sqcup n}\sqcup \ast)^{\wedge}_{C_n},\mathbb{Q}_p)\]
whose homotopy colimit is $D_*((\bmod_{0,n+1})^\wedge,\mathbb{Q}_p)$. This simplicial diagram has an action of $\gt$ that induces the action of $\gt$ on $D_*((\bmod_{0,n+1})^\wedge,\mathbb{Q}_p)$ constructed in Theorem \ref{thm: action on bmod}. We thus get a $\gt$-equivariant spectral sequence of the form
\[E^1_{s,t}=H_t((\ast\sqcup (S^1)^{\sqcup s} \sqcup \mathcal{M})(n),\mathbb{Q}_p)\implies H_{s+t}(\bmod_{0,n +1},\mathbb{Q}_p)\]
By the K\"unneth isomorphism, and Proposition \ref{prop: action on H}, we deduce that the action of $g\in \gt$ on $E^1_{s,t}$ is given by multiplication by $\chi_p(g)^t$. This implies that the same is true for $E^r_{s,t}$, for all $r$. Since the differentials must commute with the $\gt$-action and $d^r$ has degree $(-r,r-1)$, we see that only $d^1$ can be non-zero.

Now we study the behavior of the line $t=0$ of this spectral sequence. We can do that by comparing it to the similar spectral sequence denoted $E'^r_{s,t}$ that computes the pushout square of simplicial operads
\[
\xymatrix{
\ast\ar[d]\ar[r]& \mathsf{Com}\ar[d]\\
\ast\ar[r]& \mathsf{Com}
}
\]
We have a map of spectral sequences $E\to E'$ that comes from a map between the two pushout squares. This map is an isomorphism on the line $t=0$ at the $E^1$ page. The spectral sequence $E'$ is very easy to understand: the differential $d^1$ kills everything except the generator in degree $(0,0)$ and no further differentials can occur. Therefore, the same patterns must occur on the $0$-th line of the spectral sequence $E$.

We have said that $E$ collapses at the $E^2$ page. Using the observation that $E^2_{s,0}$ is trivial for $s>0$, we deduce that for positive $k$ the vector space $H_k((\bmod_{0,n+1})^{\wedge},\mathbb{Q}_p)$ has a filtration which is compatible with the $\gt$ action and that the action of $\gt$ on the associated graded splits as a direct sum of representations that are non-trivial.
\end{proof}



\section{Unital case}\label{sec:unital}

In this section, we allow operads to have \emph{non-empty} space of $0$-arity operations. For us the most relevant such operad is $\mathcal{M}_+$, the variant of the genus zero surface operad where $\mathcal{M}_+(0) = *$. The essential difference between $\mathcal{M}$ and $\mathcal{M}_+$ is that in the latter we include the operation of filling in boundary components. While this is a substantial difference, we show that it does not affect the monoid of derived endomorphisms:

\begin{thm}\label{thm:mainunital}
$\ugt \cong \pi_0 \REnd(\widehat{N\mathcal{M}_+})$
\end{thm}

There are two relevant unital variants of the operad $\parb$. The first one is an operad $\parb_+$ which coincides with $\parb$ in positive arities and is a point in arity zero. The operadic composition maps $\circ_i$ of the form
\[
\parb_+(n) \times \parb_+(0) \to \parb_+(n-1)
\]
are given (on morphism sets) by removing the $i^{th}$ strand. We do not expect the operad $\parb_+$ to be cofibrant in any reasonable sense. So we introduce a second operad, denoted $\parb_\star$, which is cofibrant as a monochromatic operad in groupoids and is equivalent to $\parb_+$. In detail, the operad $\ob \parb_\star$ of objects of $\parb_\star$ is the free operad on the (non-symmetric) sequence which is a point in degree $0$ and $2$ and empty otherwise. We think of an element in $\parb_\star(n)$ as a parenthesization of a word $a_1 \dots a_\ell$ where each $i \in \{1, \dots,n\}$ occurs exactly once as one of the $a_j's$ and all the other symbols are labelled $*$. (Alternatively, that element may be regarded as a tree with $n$ leaves where the vertices have either two inputs or no inputs.) For example, $((*(1*))((**)(32)))$ is an element in $\ob \parb_\star(3)$.

There is a canonical map of operads $u : \ob \parb_\star \to \ob \parb_+$ that drops the symbols $*$ (and the appropriate parenthesis). For instance, the element in the example just above is sent to $(1(32))$ via this map. 

Given two objects $x$ and $y$ in $\parb_\star$, the set of morphisms in $\parb_\star$ from $x$ to $y$ is by definition the set of morphisms in $\parb_+$ from $u(x)$ to $u(y)$. This defines the operad $\mor \parb_\star$ and thus the operad $\parb_\star$ together with a map $v:\parb_\star\to\parb_+$.

\begin{prop}
The map $v$ is a cofibrant replacement of $\parb_+$ in $\Op(\G)$.
\end{prop}

\begin{proof}
This map is clearly a levelwise weak equivalence. Moreover, since the operad $\ob\parb_\star$ is freely generated by an operation in degree $0$ and an operation in degree $2$, we can apply Proposition \ref{prop: free op gp} and conclude that $\parb_\star$ is cofibrant.
\end{proof}

The following two lemmas are variations on Lemma \ref{lem:maps_pab} and Lemma \ref{lem:maps_parb}.

\begin{lemma}\label{lem:maps_pabplus} Let $\mathsf{P}$ be an operad in groupoids with $\mathsf{P}(0) = *$ and $\ob\mathsf{P}(1) = *$. The set of operad maps from $\pab_+$ to $\mathsf{P}$ is identified with the set of maps $g = (m,\beta, \alpha) : \pab \to \mathsf{P}$ subject to the relation $\alpha \circ_i id_* = id_{m}$ for $i = 1,2,3$, where $id_*$ denotes the identity element of $* \in \mathsf{P}(0)$.
\end{lemma}
\begin{proof}
See \cite[6.2.4(c)]{FresseBook1}.
\end{proof}

\begin{lemma}\label{lem:maps_parbplus}
Let $\mathsf{P}$ be an operad in groupoids with $\mathsf{P}(0) = *$. The set of operad maps from $\parb_+$ to $\mathsf{P}$ is identified with the set of pairs $(g, \tau)$ where $g = (m,\beta,\alpha)$ is an operad map from $\pab_+$ to $\mathsf{P}$ and $\tau$ is a morphism in $\mathsf{P}(1)$, subject to the relation that the operadic composition $\tau \circ_1 id$ agrees with the categorical composition $\beta \cdot \sigma\beta \cdot (id \circ (\tau,\tau))$.
\end{lemma}

\begin{proof}
The proof of Lemma \ref{lem:maps_parb} applies verbatim, but one extra check needs to be completed, namely that
\[
f([n_1, \dots, n_k] \circ_i id_*) = f([n_1, \dots, n_k]) \circ_i f(id_*) 
\]
where $id_*$ denotes the identity map of the (unique) element in arity $0$. This can be reduced to checking that $f([1] \circ id_*) = f([1]) \circ_i f(id_*)$, which holds automatically because $\mathsf{P}(0) = *$.
\end{proof}

\begin{prop}\label{prop: plus end to ho}
The restriction map $\End_0(\parb_+) \to \End_0(\parb)$ is an isomorphism.
\end{prop}

\begin{proof}
Exactly as in Proposition \ref{prop: endomorphisms fixing the objects}, we can prove that any endomorphism of $\parb_+$ fixing the objects has to restrict to an endomorphism of the suboperad $\pab_+$. This is by definition the suboperad with the same objects but only those morphisms that have trivial twists. We thus get a commutative square of restriction maps 
\[
\xymatrix{
\End_0(\parb_+)\ar[r]\ar[d]&\End_0(\pab_+)\ar[d]\\
\End_0(\parb )\ar[r]& \End_0(\pab)
}
\]
the lower horizontal map is an isomorphism by Lemma \ref{prop: endomorphisms fixing the objects} and the top horizontal map is an isomorphism by Lemma \ref{lem:maps_parbplus}. In order to see that the right-hand vertical map is an isomorphism it is enough, by Lemma \ref{lem:maps_pabplus}, to show that for any map $g : \pab \to \pab_+$ fixing the objects, the equation $\alpha \circ_i id_* = id_{m}$ automatically holds. This equation boils down to 
the condition that the image of $\alpha$, viewed as an element in the pure braid group on three strands, under each of the three maps $\partial_i : \pb(3) \to \pb(2) \cong \mathbb{Z}$ which forgets the $i^{th}$ strand, is zero. We can use one of the hexagon relations to deduce this. We explain this for $i = 2$ (the middle strand), the other cases can be treated similarly. Recall that $\beta \in \pab_+(2)$ is by definition the image of $\boldsymbol{\beta} \in \pab(2)$ under $g$. The hexagon relation (\ref{hexigon_rel_1}) reads
\begin{equation}\label{eq:hexagain}
(m \circ_1 \beta) \cdot (213)\alpha \cdot (213)(m\circ_2 \beta) = \alpha \cdot (\beta \circ_2 m ) \cdot (231)\alpha
\end{equation}
in $\pb(3)$. (In this equation, $m$ is short-hand for $id_m$.) By inspection, we have that
\[
\partial_2(m \circ_1 \beta) = 0 \; , \quad \partial_2((213)m \circ_2 \beta) = \beta \quad and \quad  \partial_2 (\beta \circ_2 m ) = \beta
\]
in $\pb(2)$. (To verify these it may be helpful to draw a picture.) Write $n$ for $\partial_2(\alpha) = \partial_2((231)\alpha)$. Now apply $\partial_2$ to both sides of (\ref{eq:hexagain}) to deduce that $0 + n + \beta = n + \beta + n$. Hence, $\partial_2(\alpha) = 0$ as claimed.
\end{proof}

\begin{cor}\label{coro: gt is endo unital}
The action of $\ugt$ on $\parb$ extends to an action on $\parb_+$ and the induced map
\[\ugt\to\End_0(\parb_+)\]
is an isomorphism.
\end{cor}

In analogy with previous notations, we write $\Hom_0(\parb_\star, \parb_+)$ for the set of operad maps which induce the map $u : \ob \parb_\star \to \ob \parb_+$ on objects.

\begin{lemma}\label{lem:stupidhardlemma}
The map $\End_0(\parb_+) \to \Hom_0(\parb_\star, \parb_+)$ is an isomorphism.
\end{lemma}

\begin{proof}
This map is injective since the map $v:\parb_\star\to\parb_+$ is an epimorphism in the category of operads in groupoids.

Before proving surjectivity, we start by making the observation that the groupoid $\parb_+(n)$ sits naturally inside $\parb_\star(n)$ as the full subgroupoid spanned by those objects that do not have the symbol $\star$. One should observe that these maps do not assemble into a map of operads $\parb_+\to\parb_\star$. We will use these maps implicitly to see morphisms in $\parb_+$ as morphisms in $\parb_\star$ when needed. 

In order to prove the surjectivity of the map under consideration, we first observe that its image is the set of operad maps $f:\parb_\star\to\parb_+$ inducing the map $u$ on objects and with the property that $fx=fvx$ for any morphism $x$ in $\parb_\star$. Now, we observe that any morphism $x$ in $\parb_\star$ can be written as a composition $a\cdot vx\cdot b$ where $a$ and $b$ are morphisms in $\pab_\star\subset \parb_\star$ that are such that $ua$ and $ub$ are identity maps. Hence, it suffices to prove that $f(a)$ and $f(b)$ are identity morphisms. But, one can prove exactly as in Proposition \ref{prop: endomorphisms fixing the objects} that $f$ has to restrict to a map of operads $\pab_\star\to\pab_+$. Moreover, by \cite[Lemma 6.6.]{Horel1}, any map of operads $\pab_\star\to\pab_+$ lies in the image of the map
\[\Hom_0(\pab_+,\pab_+)\to\Hom_0(\pab_\star,\pab_+)\]
Therefore $f(a)$ and $f(b)$ are identity morphisms as desired.
\end{proof}

\begin{lemma}\label{lem: easylemma}
The map $$\Hom_0(\parb_\star, \parb_+) \to \hoHom(\parb_\star, \parb_+)$$ is an isomorphism.
\end{lemma}

\begin{proof}
Denote the map under consideration by $\alpha$. We begin by proving injectivity. Unlike the inclusion of $\parb_+$ in $\parb_\star$, the inclusion of $\parb$ in $\parb_\star$ is a map of operads. Hence, by pre-composition, we obtain a map
\[
\hoHom(\parb_\star, \parb_+) \to \hoHom(\parb, \parb_+) \cong \hoEnd(\parb)
\]

that we call $\theta$. Using Lemma \ref{lem:stupidhardlemma} and Proposition \ref{prop: plus end to ho}, the map $\theta \circ \alpha$ is identified with the obvious map
\[
\End_0(\parb) \to \hoEnd(\parb) \; ,
\]
which is an isomorphism by Proposition \ref{prop:0toho}. It follows that $\alpha$ is injective.

In order to prove surjectivity, we have to show that any morphism $f$ from $\parb_\star$ to $\parb_+$ is homotopic to one which induces the map $v$ on objects. Since $\ob \parb_\star$ is freely generated by $*$ and $(12)$ in degrees 0 and 2, and since $\parb_+(0)$ is a point, $\ob f$ is determined by the image of $(12)$. This image is either $(12)$, in which case $\ob f = v$, or $(21)$. In the second case, a homotopy of $f$ with the required property can be constructed as in \cite[Theorem 7.8]{Horel1}.
\end{proof}

\begin{proof}[Proof of Theorem \ref{thm:mainunital}]
We first have an isomorphism
\[\pi_0\REnd(\widehat{N\mathcal{M}}_+)\simeq\pi_0\REnd(BN\widehat{\parb}_+)\]
which comes from the fact that the map
\[(B\parb_+(n))^{\wedge}\to B(\parb_+(n))^{\wedge}\]
is an equivalence (by goodness of the pure ribbon braid groups as in Lemma \ref{lemm: goodness of pure ribbon braid groups}). By full faithfulness of the functor $B$, we are reduced to proving that the action of $\ugt$ on $\parb_+$ induces an isomorphism
\[\ugt\cong \pi_0\REnd(N\widehat{\parb}_+) \; .\]
Corollary \ref{coro: gt is endo unital}, Lemma \ref{lem:stupidhardlemma} and Lemma \ref{lem: easylemma} also hold if we replace $\parb$ and its variants by their profinite completion. Hence, we deduce an isomorphism
\[\ugt\cong \hoHom(\widehat{\parb}_\star,\widehat{\parb}_+)\]
By adjunction, this gives us an isomorphism
\[\ugt\cong\hoHom(\parb_\star,|\widehat{\parb}_+|) \cong\pi_0\Rmap(\parb_\star,|\widehat{\parb}_+|) \]
where the second isomorphism comes from the fact that $\parb_\star$ is cofibrant. Since $\parb_\star\to\parb_+$ is an equivalence, we see that the action of $\ugt$ on $\parb_+$ induces an isomorphism
\[\ugt\cong \pi_0\Rmap(\parb_+,\widehat{\parb}_+)\]
But we have an isomorphism
\[\pi_0\Rmap(\parb_+,|\widehat{\parb}_+|)\cong \pi_0\Rmap(N\parb+,|N\widehat{\parb}_+|)\]
coming from the fact that $N:\Op\G\to\Op_\infty\G$ is fully faithful and an isomorphism
\[\pi_0\Rmap(N\parb_+,|N\widehat{\parb}_+|)\cong \pi_0\Rmap((N\parb_+)^{\wedge},(N\parb_+)^{\wedge})\]
by the derived adjunction between $\infty$-operads in profinite groupoids and $\infty$-operads in groupoids.
Putting everything together we deduce the desired result.
\end{proof}

\end{document}